\documentclass[12pt]{amsart}

\usepackage[matrix,arrow,curve,frame]{xy}
\usepackage{amsmath,amsthm,amssymb,enumerate}
\usepackage{latexsym}
\usepackage{amscd}
\usepackage[colorlinks=false]{hyperref}
\usepackage[mathscr]{eucal}
\usepackage{color}

\newtheorem{thm}[subsection]{Theorem}

\newtheorem{prop}[subsection]{Proposition}

\newtheorem{cor}[subsection]{Corollary}
\newtheorem{lemma}[subsection]{Lemma}
\newtheorem{conj}[subsection]{Conjecture}
\newtheorem{remark}[subsection]{Remark}

\theoremstyle{definition}

\numberwithin{equation}{section}

\def\C{\mathbb {C}}

\def\Z{\mathbb Z}
\def\Q{\mathbb Q}

\def\phi{{\varphi}}

\def\ra{\rightarrow}

\def\cA{{\mathcal A}}
\def\cB{{\mathcal B}}
\def\cC{{\mathcal C}}
\def\cD{{\mathcal D}}
\def\cE{{\mathcal E}}
\def\cF{{\mathcal F}}
\def\cG{{\mathcal G}}
\def\cH{{\mathcal H}}
\def\cI{{\mathcal I}}
\def\cJ{{\mathcal J}}

\def\cO{{\mathcal O}}

\def\cS{{\mathcal S}}
\def\cT{{\mathcal T}}

\def\cV{{\mathcal V}}
\def\cW{{\mathcal W}}

\def\gg{{\mathfrak g}}

\def\gl{{\mathfrak l}}

\def\go{{\mathfrak o}}
\def\gp{{\mathfrak p}}

\def\gs{{\mathfrak s}}

\newcommand{\mf}{\mathfrak}
\newcommand{\on}{\operatorname}

\newcommand{\fing}{\mf{g}}

\newfont{\german}{eufm10}

\begin{document}
\pagestyle{plain}

\title
{Orbifolds and cosets of minimal $\cW$-algebras}

\author{T. Arakawa}
\address{Research Institute for Mathematical Sciences, Kyoto University}
\email{arakawa@kurims.kyoto-u.ac.jp}
\thanks{T. A. is supported by JSPS KAKENHI Grants
\#25287004 and \#26610006}

\author{T. Creutzig} 
\address{University of Alberta}
\email{creutzig@ualberta.ca}
\thanks{T. C. is supported by NSERC Discovery Grant \#RES0019997}

\author{K. Kawasetsu} 
\address{Institute of Mathematics, Academia Sinica}
\email{kawasetsu@gate.sinica.edu.tw}
\thanks{K. K. is partially supported by JSPS KAKENHI Grant \#14J09236.}

\author{A. Linshaw} 
\address{University of Denver}
\email{andrew.linshaw@du.edu}
\thanks{A. L. is supported by Simons Foundation Grant \#318755}

\begin{abstract} Let $\gg$ be a simple, finite-dimensional Lie (super)algebra equipped with an embedding of $\gs\gl_2$ inducing the minimal gradation on $\gg$. The corresponding minimal $\cW$-algebra $\cW^k(\gg, e_{-\theta})$ introduced by Kac and Wakimoto has strong generators in weights $1,2,3/2$, and all operator product expansions are known explicitly. The weight one subspace generates an affine vertex (super)algebra $V^{k'}(\gg^{\natural})$ where $\gg^{\natural} \subset \gg$ denotes the centralizer of $\gs\gl_2$. Therefore $\cW^k(\gg, e_{-\theta})$ has an action of a connected Lie group $G^{\natural}_0$ with Lie algebra $\gg^{\natural}_0$, where $\gg^{\natural}_0$ denotes the even part of $\gg^{\natural}$. We show that for any reductive subgroup $G \subset G^{\natural}_0$, and for any reductive Lie algebra $\gg' \subset \gg^{\natural}$, the orbifold $\cO^k = \cW^k(\gg, e_{-\theta})^{G}$ and the coset $\cC^k = \text{Com}(V(\gg'),\cW^k(\gg, e_{-\theta}))$ are strongly finitely generated for generic values of $k$. Here $V(\gg')$ denotes the affine vertex algebra associated to $\gg'$. We find explicit minimal strong generating sets for $\cC^k$ when $\gg' = \gg^{\natural}$ and $\gg$ is either $\gs\gl_n$, $\gs\gp_{2n}$, $\gs\gl(2|n)$ for $n\neq 2$, $\gp\gs\gl(2|2)$, or $\go\gs\gp(1|4)$. Finally, we conjecture some surprising coincidences among families of cosets $\cC_k$ which are the simple quotients of $\cC^k$, and we prove several cases of our conjecture. \end{abstract}

\maketitle
\section{Introduction}
Vertex algebras arose out of conformal field theory in the 1980s and have been developed mathematically from various points of view in the literature \cite{B,FLM,K}. In addition to several well-known families such as free field algebras, affine vertex algebras, lattice vertex algebras and $\cW$-algebras, there are some standard ways to construct new vertex algebras from old ones. The {\it orbifold construction} begins with a vertex algebra $\cV$ and a group of automorphisms $G$, and considers the invariant subalgebra $\cV^G$ and its extensions. Similarly, the {\it coset construction} associates to a vertex algebra $\cV$ and a subalgebra $\cA$, the subalgebra $\text{Com}(\cA,\cV) \subset \cV$ which commutes with $\cA$ \cite{FZ}. It is believed that nice properties of $\cV$ such as strong finite generation, $C_2$-cofiniteness, and rationality, will be inherited by $\cV^G$ and $\text{Com}(\cA,\cV)$ if $G$ and $\cA$ are also nice, but few general results of this kind are known.

The structure of orbifolds of free field algebras has been studied in a series of papers using ideas from classical invariant theory \cite{LI,LII,LIII,LIV,LV,CLIII}. The main result is that for a free field algebra $\cF$ (either a Heisenberg algebra, free fermion algebra, symplectic fermion algebra, or $\beta\gamma$-system), and a reductive group $G$ of automorphisms of $\cF$, the orbifold $\cF^G$ is strongly finitely generated. The proof is essentially constructive and proceeds by first finding an explicit strong generating set for $\cF^{\text{Aut}(\cF)}$. Here $\text{Aut}(\cF)$ denotes the full automorphism group, which is either an orthogonal or symplectic group. The description of $\cF^{\text{Aut}(\cF)}$ is a formal consequence of Weyl's first and second fundamental theorems of invariant theory for these groups \cite{W}. The strong finite generation of $\cF^{\text{Aut}(\cF)}$ implies that all $\cF^{\text{Aut}(\cF)}$-modules appearing in the decomposition of $\cF$ have the $C_1$-cofiniteness property according to Miyamoto's definition \cite{Mi}. This fact, together with a classical result (Theorem 2.5A of \cite{W}), allows us to deduce the strong finite generation of $\cF^G$ for a general reductive group $G \subset \text{Aut}(\cF)$. 

It was observed in \cite{LIV} that similar results for orbifolds of affine vertex algebras at generic level can be obtained by a deformation process. A large family of cosets of affine vertex algebras inside larger structures can also be studied using this process. Let $\cA^k$ be a family of vertex algebras whose structure constants depend continuously on a complex parameter $k$. If $\cA^k$ contains a universal affine vertex algebra $V^k(\gg)$ of some simple Lie algebra $\gg$ at level $k$, let $$\cC^k = \text{Com}\left(V^k(\gg), \cA^k\right).$$
The problem of describing $\cC^k$ can be formulated and answered using the language of {\it deformable families} of vertex algebras. A deformable family $\cA$ associated to $\cA^k$ is roughly speaking a vertex algebra defined over the ring $F_K$. Here $K$ is a subset of $\mathbb{C}$ which is at most countable, and $F_K$ is the ring of complex-valued rational functions in a variable $\kappa$ of degree at most zero, with possible poles along $K$. One requires that $\cA/(\kappa-k)\cong \cA^k$ for $k\notin K$. This notion is due to \cite{CLI, CLII} and it allows us to find strong generators of $\cC^k$ for generic $k$ by passing to the limit $\lim_{k\rightarrow \infty} \cC^k$. This limit can then be identified with an orbifold of a free field algebra, which can be studied using the above methods. 

If $T = \{a_1,\dots, a_r\}$ is a strong generating set for $\text{Com}(V^k(\gg), \cA^k)$ for generic values of $k$, the set of {\it nongeneric} values of $k$, where $T$ fails to strongly generate $\text{Com}(V^k(\gg), \cA^k)$, can be described as follows. It contains a set of rational numbers $k$ with $k+ h^{\vee} \leq 0$ where $h^{\vee}$ is the dual Coxeter number of $\gg$, and has at most finitely many additional elements. These arise as poles of the structure constants appearing in the OPEs of $a_i(z) a_j(w)$, for $i,j = 1,\dots, r$. In particular, there are at most finitely many real numbers $k > -h^{\vee}$ that are nongeneric, and they can be determined from the OPE algebra of the generators. This method can be easily modified to handle the case when $\gg$ is reductive, i.e., a sum of simple and abelian ideals; see Corollary 6.11 of \cite{CLII}.

\subsection{Minimal $\cW$-algebras}
We shall adapt these methods to study orbifolds and cosets of the minimal $\cW$-algebra $\cW^k(\gg,e_{-\theta})$ introduced by Kac and Wakimoto in \cite{KWI}. Here $\gg$ is a simple, finite-dimensional Lie (super)algebra equipped with an $\gs\gl_2$-triple $\{e,f,h\}$ inducing the minimal $\frac{1}{2} \mathbb{Z}$-gradation 
$$\gg = \gg_{-1} \oplus \gg_{-1/2} \oplus \gg_0 \oplus \gg_{1/2} \oplus \gg_1.$$ In this notation, $\gg_{-1} = \mathbb{C} f$ and $\gg_1 = \mathbb{C} e$, and the above decomposition is the eigenvalue decomposition with respect to $\text{ad}(h)$ where $h = [e,f]$. Define $\gg^{\natural}$ to be the centralizer of $\{e,f,h\}$ in $\gg$. By Theorem 5.1 of \cite{KWI}, $\cW^k(\gg, e_{-\theta})$ is strongly generated by a Virasoro field $T$, primary weight one fields $J^a$, $a\in \gg^{\natural}$, which generate an affine vertex (super)algebra $V^{k'}(\gg^{\natural})$, and primary weight $3/2$ fields $G^u$, $u \in \gg_{-1/2}$. Moreover, all OPE relations are given explicitly in this theorem. 

In Section \ref{sect:defandlim}, we show that $\lim_{k\ra \infty} \cW^k(\gg,e_{-\theta})$ is a {\it generalized free field algebra} in an appropriate sense. The orbifold theory of generalized free field algebras can be studied using classical invariant theory, and orbifolds and cosets of minimal $\cW$-algebras for generic values of $k$ can be studied by passing to the limit. The even part $\gg_0^{\natural} \subset \gg^{\natural}$ is reductive and $\cW^k(\gg, e_{-\theta})$ decomposes into a sum of finite-dimensional $\gg_0^{\natural}$-modules. Therefore the action of $\gg_0^{\natural}$ lifts to an action of a choice of connected Lie group $G^{\natural}_0$ on $\cW^k(\gg, e_{-\theta})$ by automorphisms.

\begin{thm} Let $G \subset G^{\natural}_0$ be a reductive group of automorphisms of $\cW^k(\gg, e_{-\theta})$. Then the orbifold $\cW^k(\gg, e_{-\theta})^G$ is strongly finitely generated for all but finitely many values of $k$.
\end{thm}

\begin{thm} Let $\gg' \subset \gg^{\natural}$ be a reductive Lie subalgebra and let $V(\gg')\subset V^{k'}(\gg^{\natural})$ be the corresponding affine vertex subalgebra. Then the coset $\cC^k =  \text{Com}(V(\gg'), \cW^k(\gg, e_{-\theta}))$ is strongly finitely generated for generic values of $k$.
\end{thm}

We are mostly interested in the case $\gg' = \gg^{\natural}$ and $V(\gg') = V^{k'}(\gg^{\natural})$, which we consider in Section \ref{sect:genericbehavior}. We give explicit minimal strong generating sets for $\cC^k$ when $\gg$ is either $\gs\gl_n$, $\gs\gp_{2n}$, $\gs\gl(2|n)$ for $n\neq 2$, $\gp\gs\gl(2|2)$, or $\go\gs\gp(1|4)$.

\begin{thm} \label{intro:genericstructures} 
For generic values of $k$, we have the following.
\begin{itemize}
\item[(a)] For $n\geq 3$, $\text{Com}(V^{k+1}(\gg\gl_{n-2}),\cW^k(\gs\gl_n, e_{-\theta}))$ is of type $\cW(2,3,\dots, n^2-2)$.
\item[(b)] For $n\geq 2$, $\text{Com}(V^{k+1/2}(\gs\gp_{2n-2}),\cW^k(\gs\gp_{2n}, e_{-\theta}))$ is of type $\cW(2,4,\dots, 2n^2+2n-2)$.
\item[(c)] For $n\neq 2$, $\text{Com}(V^{k-1}(\gg\gl_n), \cW^k(\gs\gl(2|n), e_{-\theta}))$ is even and of type $\cW(2,3,\dots, 3n+2)$.
\item[(d)] $\text{Com}(V^{k-1}(\gs\gl_2), \cW^k(\gp\gs\gl(2|2), e_{-\theta}))$ is even and of type $\cW(2,3^3,4, 5^3, 6, 7^3, 8)$.
\item[(e)] $\text{Com}(V^{k+1/2}(\go\gs\gp(1|2)), \cW^k(\go\gs\gp(1|4), e_{-\theta}))$ is even and of type $\cW(2,4,\dots, 18)$.
\end{itemize}
\end{thm}

\subsection{The nongeneric set}

Suppose that $\cC^k = \text{Com}(V^{k'}(\gg^{\natural}), \cW^k(\gg, e_{-\theta}))$ has a strong generating set $T = \{a_1\dots, a_r\}$ for generic values of $k$. Recall that the {\it nongeneric} set of values of $k$, where $T$ fails to strongly generate $\cC^k$, consists of the union of a set of negative rational numbers and a finite set. This is not quite explicit enough for our purposes, the goal being to describe $\cC^k$ at special values of $k$ for which it is highly reducible. It was shown in \cite{ACL} that $\text{Com}(\cH,\cW^k(\gs\gl_3, e_{-\theta}))$ is of type $\cW(2,3,4,5,6,7)$ for all $k$ except for $-1$ and $-3/2$. However, this example is somewhat degenerate because cosets of Heisenberg algebras are easier to understand than the general case. In Section \ref{sect:nongeneric}, we analyze the structure of the nongeneric set for $\text{Com}(V^{k'}(\gg^{\natural}), \cW^k(\gg, e_{-\theta}))$ in the cases $\gg = \gs\gl_4$ and $\gg = \gs\gp_4$. For our applications to the structure of simple cosets, it is fruitful to first find an {\it infinite} strong generating set with the property that we can determine the poles of the structure constants appearing in the OPEs of the generators.

\begin{thm} \label{intro:nongeneric} 
\begin{itemize}
\item[(a)] For all real $k>-5/2$, $\text{Com}(V^{k+1/2}(\gs\gp_{2}),\cW^k(\gs\gp_{4}, e_{-\theta}))$ has an infinite strong generating set consisting of a field in each weight $2,4,6,\dots$. Moreover, the generators in weights $w \geq 12$ can be eliminated, so $\text{Com}(V^{k+1/2}(\gs\gp_{2}),\cW^k(\gs\gp_{4}, e_{-\theta}))$ is of type $\cW(2,4,6,8,10)$.
\item[(b)] For all real $k>-3$,  $\text{Com}(V^{k+1}(\gg\gl_{2}),\cW^k(\gs\gl_4, e_{-\theta}))$ has an infinite strong generating set consisting of a field in each weight $2,3,4,\dots$. For all but finitely many real numbers $k>-3$, the generators in weights $w \geq 15$ can be eliminated, so $\text{Com}(V^{k+1}(\gg\gl_{2}),\cW^k(\gs\gl_4, e_{-\theta}))$ is of type $\cW(2,3,4,\dots, 14)$. 
\end{itemize}
\end{thm}
The proof of this result depends on computer calculations performed using the Mathematica package of Thielemans \cite{T}. In general, we expect that the most interesting values of $k$, where $\cW^k(\gg, e_{-\theta})$ is highly reducibly, are generic. This holds in the above examples $\gg = \gs\gl_4$ and $\gg = \gs\gp_4$, but a conceptual rather than computational approach to this problem is currently out of reach. 

\subsection{Simple cosets}
The main reason to study $\cC^k = \text{Com}(V^{k'}(\gg^{\natural}),\cW^k(\gg, e_{-\theta}))$ is for the purpose of studying the coset $$\cC_k = \text{Com}(L_{k'}(\gg^{\natural}),\cW_k(\gg, e_{-\theta}))$$ at special values of $k$ where the maximal proper ideal $\cI_k \subset \cW^k(\gg, e_{-\theta})$ is large. Here $L_{k'}(\gg^{\natural})$ denotes the simple affine vertex algebra and $\cW_k(\gg, e_{-\theta})$ is the simple quotient $\cW^k(\gg, e_{-\theta})/ \cI_k$. For example, when $k+1/2$ is a non-negative integer, $\cW_k(\gs\gp_{2n}, e_{-\theta})$ is $C_2$-cofinite \cite{ArIII}, and conjecturally rational. It contains a copy of $L_{k+1/2}(\gs\gp_{2n-2})$, so $$\cC_k = \text{Com}(L_{k+1/2}(\gs\gp_{2n-2}),\cW_k(\gs\gp_{2n}, e_{-\theta}))$$ is also expected to be rational and $C_2$-cofinite. One would like to understand the structure of $\cC_k$ as well as the structure of $\cW_k(\gg, e_{-\theta})$ as an extension of $L_{k+1/2}(\gs\gp_{2n-2}) \otimes \cC_k$.

Conjecturally  \cite{AM1},
for $n\geq 4$, $\cW_k(\gs\gl_n, e_{-\theta})$ is not $C_2$-cofinite for any $k \in \mathbb{N}$. However, $\cW_k(\gs\gl_n, e_{-\theta})$ contains a copy of $L_{k+1}(\gg\gl_{n-2}):=\cH \otimes L_{k+1}(\gs\gl_{n-2})$, and is expected to be  \lq\lq small" in the sense that it has a one-dimensional associated variety \cite{Ara12}. One expects that
$$\text{Com}(L_{k+1}(\gg\gl_{n-2}),\cW_k(\gs\gl_n, e_{-\theta})),\qquad k \in \mathbb{N},$$ also has a one-dimensional associated variety. As above, we would like to describe $\cC_k$ as well as $\cW_k(\gg, e_{-\theta})$ as an extension of $L_{k+1}(\gg\gl_{n-2}) \otimes \cC_k$.

In all these cases, the natural map $\pi_k: \cC^k \ra \cC_k$ is surjective and $\cC_k$ coincides with the simple quotient of $\cC^k$. Therefore a strong generating set for $\cC^k$ descends to a strong generating set for $\cC_k$. This illustrates the importance of both finding a strong generating set for $\cC^k$ for generic $k$, and determining the nongeneric set.

\subsection{Some remarkable coincidences}
If $\cA^k$ is a vertex algebra containing $V^k(\gg)$ and $k$ is a special value for which the simple quotient $\cA_k$ contains $L_k(\gg)$ and has a low-dimensional associated variety, then $\cC_k = \text{Com}(L_k(\gg), \cA_k)$ also has a low-dimensional associated variety. Moreover, vertex algebras with a given strong generating set with a low-dimensional associated variety are rare. If there is another family of vertex algebras $\cD_k$ with the same central charges and strong generating sets, we take this as enough information to conjecture the existence of an isomorphism $\cC_k \cong \cD_k$. In general, such conjectures can then either be proven by direct OPE computation or by proving a uniqueness theorem for OPE-algebras with a given strong generating set. In the case where $\cA^k$ is a minimal $\cW$-algebra, we are indeed able to state such a uniqueness theorem; see Theorem \ref{uniqueness}. Hence if we can prove that $L_k(\gg) \otimes \cC_k$ allows for a vertex algebra extension that is of minimal $\cW$-algebra type that is simple, this provides a method to prove coincidences of $\cC_k$ and $\cD_k$. This is one of the many purposes to establish an effective theory of VOA extensions \cite{CKL,CKM}. Furthermore, once having $\cA_k$ as an extension of $L_k(\gg) \otimes \cC_k$, one can use the representation theory of VOA extensions to understand the one of $\cA_k$ provided the one of $L_k(\gg) \otimes \cC_k$ is known.

\begin{conj} \label{conj:typecintro} For all $n\geq 2$ and $k$ such that $k+1/2$ is a positive integer, $$ \cC_k(n) = \text{Com}(L_{k+1/2}(\gs\gp_{2n-2}), \cW_k(\gs\gp_{2n}, e_{-\theta}))$$ is isomorphic to the $C_2$-cofinite, rational principal $\cW$-algebra
$\cW_s(\gs\gp_{2m},f_{\text{prin}})$ where $$m=k+1/2,\qquad s+(k+3/2)=\frac{n+k+1/2}{2n+2k+2}.$$
\end{conj}

In the case $n=2$ and $k=1/2$, this was proven by Kawasetsu using the fact that $\cW_{1/2}(\gs\gp_4, e_{-\theta})$ is a simple current extension of $L_1(\gs\gl_2) \otimes L(-25/7,0)$ \cite{Ka}. Here $L(-25/7,0)$ denotes the simple Virasoro vertex algebra with $c = -25/7$. Since $\cW_k(\gs\gp_{2n}, e_{-\theta})$ is an extension of $L_{k+1/2}(\gs\gp_{2n-2}) \otimes \cC_k(n)$, Conjecture \ref{conj:typecintro} implies the rationality of $\cW_k(\gs\gp_{2n}, e_{-\theta})$ for all $n$ and $k$ as above. Using Theorem \ref{intro:nongeneric}, we give an alternative proof of Kawasetsu's result, and we also prove our conjecture for $n=2$ and $k=3/2$. This implies the rationality of $\cW_{3/2}(\gs\gp_4, e_{-\theta})$.

For $n\geq 3$, let $\cW^{\ell}(\gs\gl_n,f_{\text{subreg}})$ be the $\cW$-algebra associated with a subregular nilpotent element $f_{\text{subreg}}$
at level ${\ell}$ \cite{KRW}, and let
$\cW_{\ell}(\gs\gl_n,f_{\text{subreg}})$ be the unique simple quotient of $\cW^{\ell}(\gs\gl_n,f_{\text{subreg}})$.
It was recently conjectured in \cite{ACGHR}, and proven by Genra \cite{G}, that 
$\cW^{\ell}(\gs\gl_n,f_{\text{subreg}})$ is isomorphic to 
the vertex algebra
$\cW^{(2)}_{n,{\ell}}$ introduced by 
Feigin and Semikhatov \cite{FS}. Note that $\cW^{\ell}(\gs\gl_n,f_{\text{subreg}})$ is of type $\cW(1,2,\dots, n-1, n/2, n/2)$ and the weight one field generates a Heisenberg algebra $\cH$. Also, note that for $n=1$ and $n=2$, $\cW^{(2)}_{n, \ell}$ is well-defined and coincides with the rank one $\beta\gamma$-system and the affine vertex algebra $V^{\ell}(\gs\gl_2)$, respectively.

\begin{conj} \label{typeAconj-intro} For all integers $n\geq 4$ and $k\geq 0$ such that
 $n^2-k^2-1\geq n$, $$\cC_k(n) = \text{Com}(L_{k+1}(\gg\gl_{n-2}), \cW_k(\gs\gl_n, e_{-\theta}))
\cong \text{Com}(\cH, \cW_\ell(\gs\gl_{k+1},f_{\text{subreg}})),$$  where $\ell = - \frac{1 + k^2 + k n}{k + n}$.
\end{conj}

To interpret this conjecture for $k=0$ and $k=1$, we replace $\cW_{\ell}(\gs\gl_{k+1},f_{\text{subreg}})$ with the $\beta\gamma$-system and the simple affine vertex algebra $L_{\ell}(\gs\gl_2)$, respectively. It is worth mentioning that the subregular nilpotent orbit, which is the Lusztig-Spaltenstein dual of the minimal nilpotent orbit, appears in the above correspondence.

We will prove Conjecture \ref{typeAconj-intro} for all $n\geq 4$ and $k=0$ using our uniqueness result for minimal $\cW$-algebras, and for $n=4$ and $k=1$ using Theorem \ref{intro:nongeneric}. Since the simple algebra $\cW^{(2)}_{2,-6/5}$ is isomorphic to $L_{-6/5}(\gs\gl_2)$, $\cC_1(4)$ is isomorphic to the simple $\gs\gl_2$-parafermion algebra at level $-6/5$. In the case $n=4$ and $k=2$ there is also considerable computational evidence that the conjecture holds. We remark that conformal embeddings of affine vertex (super)algebras in minimal $\cW$-algebras have been studied recently in \cite{A-PI, A-PII}. Their findings can be viewed as cases of the coincidences where the simple coset $\cC_k \cong \mathbb{C}$. 

\subsection{Constructing new $C_2$-cofinite but non-rational VOAs} The representation category of grading restricted modules of a simple, CFT-type and $C_2$-cofinite VOA that is its own contragredient dual is expected to be a log-modular tensor category \cite{CGI, CGII}. Similar to its rational cousin this comes with a close and powerful relation of the modularity of (pseudo)-trace functions \cite{Miy} and (logarithmic)-Hopf links \cite{CGI, CGII}. In order to understand this relation better, one needs new examples of interesting $C_2$-cofinite but non-rational VOAs.

In some instances the coset of a minimal $\cW$-algebra might be $C_2$-cofinite or might allow for a VOA-extension to a $C_2$-cofinite VOA. Examples are our cases of $k=0$ and $n\geq 4$ of Conjecture \ref{typeAconj-intro}. Here, we always obtain for $\cC_0(n)$ the  singlet $\cW$-algebra of central charge $c=-2$ which has the triplet algebra of the same central charge as simple current extension. 
Another example is the case of $\gg=\gs\gl_3$ and $k=-9/4$. In this case the coset $\cC_k$ of the minimal $\cW$-algebra is the singlet $\cW$-algebra of central charge $c=-7$ \cite{CRW}, with the triplet algebra of the same central charge as simple current extension as well. The triplet VOAs are known to be $C_2$-cofinite \cite{AdM}. We are currently looking for new $C_2$-cofinite VOAs from cosets of minimal $\cW$-algebras.

\section{Vertex algebras}
We assume that the reader is familiar with the basics of vertex algebra theory, which has been discussed from several different points of view in the literature (see for example \cite{B,FLM,K,FBZ}). Given an element $a$ in a vertex algebra $\cV$, the corresponding field is denoted by $$a(z) = \sum_{n\in \mathbb{Z}} a(n) z^{-n-1}.$$ Given $a,b \in \cV$, the {\it operators product expansion} (OPE) formula is given by
$$a(z)b(w)\sim\sum_{n\geq 0}(a_{(n)} b)(w)\ (z-w)^{-n-1}.$$ Here $(a_{(n)} b)(w) = \text{Res}_z [a(z), b(w)](z-w)^n$ where $[a(z), b(w)] = a(z) b(w)\ -\ (-1)^{|a||b|} b(w)a(z)$, and $\sim$ means equal modulo terms which are regular at $z=w$. The {\it normally ordered product} $:a(z)b(z):$ is defined to be $$a(z)_-b(z)\ +\ (-1)^{|a||b|} b(z)a(z)_+,$$ where $$a(z)_-=\sum_{n<0}a(n)z^{-n-1},\qquad a(z)_+=\sum_{n\geq
0}a(n)z^{-n-1}.$$  For fields $a_1(z),...,a_k(z)$, the $k$-fold
iterated Wick product is defined inductively to be
\begin{equation} \label{iteratedwick} :a_1(z)a_2(z)\cdots a_k(z):\ =\ :a_1(z) \big(  :a_2(z)\cdots a_k(z): \big).\end{equation}

A subset $S=\{a_i|\ i\in I\}$ of $\cA$ {\it strongly generates} $\cA$ if $\cA$ is spanned by $$\{ :\partial^{k_1} a_{i_1}(z)\cdots \partial^{k_m} a_{i_m}(z):| \ i_1,\dots,i_m \in I,\ k_1,\dots,k_m \geq 0\}.$$ We say that $S$ {\it freely generates} $\cA$ if there are no nontrivial normally ordered polynomial relations among the generators and their derivatives. We say that $\cA$ is of type $$\cW \big((d_1)^{n_1},\dots, (d_r)^{n_r}\big)$$ if it has a minimal strong generating set consisting of $n_i$ fields in each weight $d_i$ for $i = 1,\dots,r$.

A vertex algebra $\cA$ is called {\it simple} if it has no nontrivial two-sided ideals. Here an ideal is a subspace closed under $a_{(k)}(-)$ and $(-)_{(k)} a$ for all $a\in \cA$ and $k\in \mathbb{Z}$. The following result will be needed.
The proof is based on that of Theorem 4.4 of \cite{DLY}.

\begin{lemma} 
Let $\cV$ be a simple  $\Q_{\geq 0}$-graded vertex algebra such that 
$\cV_0=\C$,
and let
 $\cW \subset \cV$ be a simple vertex subalgebra.
Suppose that $\cV$ decomposes into a direct sum of simple $\cW$-modules $W_i$, where $W_0=\cW$, and write
$$\cV =\bigoplus_i U_i\otimes W_i,$$
where $U_i$ is the multiplicity space of $W_i$. Then $U_0$ is a simple
vertex algebra.
\end{lemma}

\begin{proof} Let $u\in U_i$. Because $\cV$ is simple there exist $v_1,\dots ,v_r\in \cV$ and $n_1,\dots,
n_r\in \mathbb{Z}$, $r\geq 1$, such that $\sum_{i=1}^r(v_i)_{(n_i)}(u\otimes 1)=1\otimes 1$ 
by Corollary 4.2 of
\cite{DM1}.
Write $$v_i=\sum_i a_{i,j}\otimes b_{i,j},$$ with $a_{i,j}\in U_j$ and $b_{i,j}\in W_j$
for each $i$. Then $$\sum_{i,j} \sum_k (a_{i,j})_{(k)}u \otimes  (b_{i,j})_{(n_i-k)} 1 =1\otimes 1.$$

Since $(b_{i,j})_{(n_i-k)}1\in W_j$, by Proposition 5.4.7 of \cite{FHL},
this means that 
$$\sum_i\sum_k (a_{i,0})_{(k)}u \otimes  (b_{i,0})_{(n_i-k)}1=1\otimes 1.$$
So there exist  $i$ and $m$ such that $(a_{i,0})_{(m)}u=1$, that is, $U_0$ is simple.
\end{proof}

The situation we wish to apply this to is the following. Suppose that $\cW$ and $\cV$ are as above, and that $\cW$ is rational. Then $\cV$ is a direct sum of irreducible $\cW$-modules, and $U_0 = \text{Com}(\cW,\cV)$. We obtain

\begin{cor} \label{simplicity}
Let $\cV$ be a simple  $\Q_{\geq 0}$-graded vertex algebra such that 
$\cV_0=\C$.
If $\cW\subset \cV$ is simple and rational, then $\text{Com}(\cW,\cV)$ is simple.
\end{cor}

\section{Minimal $\cW$-algebras}
Let $\gg$ be a simple, finite-dimensional Lie (super)algebra. We recall some basic properties of the minimal $\cW$-(super)algebra $\cW^k(\gg, e_{-\theta})$ which was introduced by Kac and Wakimoto in \cite{KWI,KWII}. First, suppose that $\{e,f,h\}$ is an $\gs\gl_2$-triple in $\gg$ inducing the minimal $\frac{1}{2} \mathbb{Z}$-gradation 
$$\gg = \gg_{-1} \oplus \gg_{-1/2} \oplus \gg_0 \oplus \gg_{1/2} \oplus \gg_1.$$
Here $\gg_{-1} = \mathbb{C} f$ and $\gg_1 = \mathbb{C} e$, and the above decomposition is the eigenvalue decomposition with respect to $\text{ad}(h)$ where $h = [e,f]$. Define $\gg^{\natural}$ to be the centralizer of $\{e,f,h\}$ in $\gg$. Recall that $\gg_{-1/2}$ and $\gg_{1/2}$ are isomorphic as $\gg^{\natural}$-modules, and that $\gg_{1/2}$ possesses a nondegenerate, skew-supersymmetric bilinear form $\langle \cdot, \cdot \rangle_{\text{ne}}$ given by $$\langle a, b \rangle_{\text{ne}} = (f|[a,b]),$$ where $(\cdot | \cdot)$ denotes the bilinear form on $\gg$.

By Theorem 5.1 of \cite{KWI}, $\cW^k(\gg, e_{-\theta})$ is strongly generated by a Virasoro field $T$, primary weight one fields $J^a$, $a\in \gg^{\natural}$, which generate an affine vertex (super)algebra $V^{k'}(\gg^{\natural})$, and primary weight $3/2$ fields $G^u$, $u \in \gg_{-1/2}$. Moreover, all OPE relations are given explicitly in this theorem. A number of well-known vertex algebras can be realized in this way, such as the Virasoro algebra, $N=2$ superconformal algebra, small $N=4$ superconformal algebra, Bershadsky-Polyakov algebra, etc.

The even part $\gg_0^{\natural} \subset \gg^{\natural}$ is reductive and $\cW^k(\gg, e_{-\theta})$ decomposes into a sum of finite-dimensional $\gg_0^{\natural}$-modules. Therefore the action of $\gg_0^{\natural}$ lifts to an action of a connected Lie group $G^{\natural}_0$ on $\cW^k(\gg, e_{-\theta})$ by automorphisms. In this paper, we are interested in the structure of orbifolds $\cW^k(\gg, e_{-\theta})^G$ where $G \subset G^{\natural}_0$ is a reductive subgroup, and cosets $\text{Com}(V(\gg'), \cW^k(\gg, e_{-\theta}))$ where $V(\gg')\subset V^{k'}(\gg^{\natural})$ is the affine vertex algebra corresponding to a reductive Lie algebra $\gg' \subset \gg^{\natural}$. 

\subsection{Uniqueness of minimal $\cW$-algebras}

Let $\cW^k(\gg,e_{-\theta})$ and $V^{k'}(\gg^{\natural})$ be as above, and let $\cW_k(\gg,e_{-\theta})$ denote the unique simple quotient of $\cW^k(\gg,e_{-\theta})$. Let $\cA_k$ be another vertex (super)algebra with the following properties: 
\begin{enumerate}
\item The central charges of $\cW^k(\gg,e_{-\theta})$ and $\cA_k$ coincide.
\item The full affine sub(super)algebra of $\cA_k$ is a homomorphic image of $V^{k'}(\gg^{\natural})$.
\item $\cA_k$ is strongly generated by the Virasoro field $T$, the weight one and the weight $3/2$ fields and the number of weight $3/2$ strong generators of $\cA_k$ and  $\cW^k(\gg,e_{-\theta})$ are the same, and of the same parity.

\item The weight $3/2$ fields in both $\cA_k$ and $\cW^k(\gg,e_{-\theta})$ have the same OPE with the Virasoro field and the weight one fields. In particular, they carry the same representation of $\gg^{\natural}$.
\end{enumerate}

\begin{thm} \label{uniqueness} Let $\cW^k(\gg,e_{-\theta})$ and $\cA_k$ be as above. Then the OPE algebras of the strong generators of 
$\cW^k(\gg,e_{-\theta})$ and $\cA_k$ coincide up to null fields. In other words $\cA_k$ is a homomorphic image of $\cW^k(\gg,e_{-\theta})$. In particular if $\cA_k$ is simple then $\cA_k\cong \cW_k(\gg,e_{-\theta})$. 
\end{thm}

\begin{proof} First, the weight $3/2$ subspace of $\cA_k$ has a $\gg^{\natural}$-invariant supersymmetric bilinear form given by $\langle X,Y\rangle = X_{(2)} Y$. This form is nondegenerate: $\cA_k$ has an order two automorphism that acts as $(-1)$ on elements of conformal weight in $\Z+\frac{1}{2}$ and as $1$ on those of integer conformal weight.  Then $\cA_k$ decomposes as
\[
\cA_k = \cA_k^+ \oplus \cA_k^-
\]
where $\cA_k^+$ is the orbifold subVOA. 
Proposition 4.1 of \cite{CKL} applied to $\cA_k^-$ implies that this form is non-zero. Since $\gg_{1/2}$ is either irreducible or is a direct sum of two contragredient modules of $\gg^\natural$ \cite{KWI} this form is nondegenerate.

So we can rescale the weight $3/2$ fields of $\cA_k$ so that the 3rd order poles coincide with the corresponding poles in $\cW^k(\gg,e_{-\theta})$. It remains to compare the first and second order poles between the weight $3/2$ fields of $\cA_k$. For this, we need the following vertex algebra identities. For any elements $U, V, W$ in a vertex superalgebra,
\begin{equation} \label{pairingi} U_{(1)} (V_{(1)}  W) = (U_{(1)}  V)_{(1)}  W  + (-1)^{|U||V|} V_{(1)} (U_{(1)}  W) + (U_{(0)}  V)_{(2)}  W,\end{equation}
\begin{equation} \label{pairingii} 
\begin{split}
U_{(3)}  (V_{(0)}  W) =  &(U_{(3)}  V)_{(0)}  W  + (-1)^{|U||V|} V_{(0)}  (U_{(3)}  W) +\\
&+3 (U_{(2)}  V)_{(1)}  W +3 (U_{(1)} V)_{(2)}  W +  (U_{(0)}  V)_{(3)}  W.
\end{split}
\end{equation}

Take $V$ and $W$ to be weight $3/2$ fields of $\cA_k$ and take $U$ to be a weight one field. Since $V_{(1)} W$ has weight one and the bilinear form on the weight one subspace $\cA_k[1]$ is nondenenerate, it follows from \eqref{pairingi} that $V_{(1)}  W$ is uniquely determined by the known OPEs.

Next, consider the weight $2$ subspace $\cA_k[2]$. It splits as the direct sum of the affine part and the span of the Virasoro element $T$. Suppose first that $V^{k'}(\gg^{\natural})$ has no singular vectors in weight $2$ and that $L = T - L_{\gg}$ has nonzero central charge. Then $\cA_k[2]$ has a nondegenerate bilinear form $\langle X,Y \rangle = X_{(3)}  Y$. Taking $V,W$ to be weight $3/2$ fields and taking $U$ to be a weight $2$ field, it then follows from \eqref{pairingii} that $V_{(0)}  W$ is uniquely determined by the known OPEs.

Finally, suppose that the form on $\cA_k[2]$ is degenerate. Write $\cA_k[2] = A_1 \oplus A_2$ where the form is nondegenerate on $A_1$ and vanishes on $A_2$. Then $A_2$ lies in the maximal proper ideal of $\cA_k$. Applying the above argument only to vectors $U \in A_1$, we see that $V_{(0)} W$ differs from the corresponding OPE in $\cW^k(\gg,e_{-\theta})$ by elements in $A_2$. \end{proof}

\begin{remark}
De Sole proved an interesting related theorem in his Ph.D. thesis (Theorem 1.2.3 of \cite{DeS}). Namely, let $\cW$ be a vertex algebra which is strongly generated by a Virasoro field together with primary fields of weights $1$ and $3/2$. The weight $1$ fields are assumed to generate an affine vertex algebra of some reductive Lie algebra $\gg$, and the weight $3/2$ fields carry a representation $\rho$ of $\gg$. Also, the $\gg$-invariant bilinear form on the weight $3/2$ fields coming from $\rho$ is assumed to be non-degenerate. Finally, $\cW$ is assumed to admit a quasi-classical limit (Definition 5.1.5 of \cite{DeS}). Then he gives a complete list of possible $\gg$ and $\rho$, all of which correspond to minimal $\cW$-algebras. \end{remark}

\subsection{Weak filtrations}
By a {\it compatible filtration} on a vertex algebra $\cA$, we mean an increasing $\mathbb{Z}_{\geq 0}$-filtration
\begin{equation} \cA_{(0)}\subset\cA_{(1)}\subset\cA_{(2)}\subset \cdots,\qquad \cA = \bigcup_{d\geq 0}
\cA_{(d)}\end{equation} such that $\cA_{(0)} = \mathbb{C}$, and for all
$a\in \cA_{(r)}$, $b\in\cA_{(s)}$, we have
\begin{equation} \label{goodi} a_{(n)}  b\in  \bigg\{\begin{matrix}\cA_{(r+s)} & n<0 \\ \cA_{(r+s-1)} & 
n\geq 0 \end{matrix}\ . \end{equation}
We set $\cA_{(-1)} = \{0\}$, and we say that $a(z)\in\cA_{(d)}\setminus \cA_{(d-1)}$ has degree $d$. Such filtrations were introduced in \cite{LiI}, and the key property is that the associated graded algebra $$\text{gr}(\cA) = \bigoplus_{d\geq 0}\cA_{(d)}/\cA_{(d-1)}$$ is a $\mathbb{Z}_{\geq 0}$-graded associative, (super)commutative algebra with a
unit $1$ under a product induced by the Wick product. For $r\geq 1$ we have the projection \begin{equation} \phi_r: \cA_{(r)} \ra \cA_{(r)}/\cA_{(r-1)}\subset \text{gr}(\cA).\end{equation}

 Given a vertex algebra $\cA$ with such a filtration, we have the following reconstruction property. Let $\{a_i|\ i\in I\}$ be a set of generators for $\text{gr}(\cA)$ as a differential algebra, so that $\{\partial^k a_i|\ i\in I, k\geq 0\}$ generates $\text{gr}(\cA)$ as a ring. If $a_i$ is homogeneous of degree $d_i$ and $a_i(z)\in\cA_{(d_i)}$ satisfies $\phi_{d_i}(a_i(z)) = a_i$, then $\cA$ is strongly generated as a vertex algebra by $\{a_i(z)|\ i\in I\}$.

In \cite{ACL}, a {\it weak filtration} on a vertex algebra $\cA$ was defined to be an increasing $\mathbb{Z}_{\geq 0}$-filtration
\begin{equation}\label{weak} \cA_{(0)}\subset\cA_{(1)}\subset\cA_{(2)}\subset \cdots,\qquad \cA = \bigcup_{d\geq 0}
\cA_{(d)}\end{equation} such that for $a\in \cA_{(r)}$, $b\in\cA_{(s)}$, we have
\begin{equation} a_{(n)}  b \in \cA_{r+s},\qquad n\in \mathbb{Z}.\end{equation}
This condition guarantees that $\text{gr}(\cA) = \bigoplus_{d\geq 0}\cA_{(d)}/\cA_{(d-1)}$ is a vertex algebra, but it is no longer abelian in general. As above, a strong generating set for $\text{gr}(\cA)$ consisting of homogeneous elements always gives rise to a strong generating set for $\cA$.

We define an increasing filtration 
$$\cW^k_{(0)} \subset \cW^k_{(1)} \subset \cdots $$ on $\cW^k = \cW^k(\gg, e_{-\theta})$ as follows: 
$\cW^k_{(-1)} = \{0\}$, and $\cW^k_{(r)}$ is spanned by iterated Wick products of $T, J^a, G^u$ and their derivatives, such that at most $r$ of the fields $G^u$ and their derivatives appear. It is clear from the defining OPE relations \cite{KWI} that this is a weak filtration. Note that $\cW^k_{(0)} \cong \text{Vir} \ltimes V^{k'}(\gg^{\natural})$ where $\text{Vir}$ denotes the Virasoro vertex algebra with generator $T$. Note also that the associated graded algebra $$\text{gr}(\cW^k) = \bigoplus_{d \geq 0} \cW^k_{(d)} / \cW^k_{(d-1)}$$ is not abelian since it contains $\cW^k_{(0)}$ as a subalgebra, but $G^u(z) G^v(w) \sim 0$ in this algebra for all $u,v$.

\section{Deformations and limits of minimal $\cW$-algebras} \label{sect:defandlim}
In this section, we shall adapt the methods of our previous papers on orbifolds and cosets of free field and affine vertex algebras to the setting of minimal $\cW$-algebras. First, recall the notion of a {\it deformable family} of vertex algebras \cite{CLI}. Let $K \subset \mathbb{C}$ be a subset which is at most countable, and let $F_K$ denote the $\mathbb{C}$-algebra of rational functions in a formal variable $\kappa$ of the form $\frac{p(\kappa)}{q(\kappa)}$ where $\text{deg}(p) \leq \text{deg}(q)$ and the roots of $q$ lie in $K$. A {\it deformable family} will be a free $F_K$-module $\cB$ with the structure of a vertex algebra with coefficients in $F_K$.  Vertex algebras over $F_K$ are defined in the same way as vertex algebras over $\mathbb{C}$. We assume that $\cB$ possesses a $\mathbb{Z}_{\geq 0}$-grading $\cB = \bigoplus_{m\geq 0} \cB[m]$ by conformal weight where each $\cB[m]$ is a free $F_K$-module of finite rank. For $k\notin K$, we have a vertex algebra $$\cB^k = \cB / (\kappa - k),$$ where $(\kappa - k)$ is the ideal generated by $\kappa - k$. Clearly $\text{dim}_{\mathbb{C}}(\cB^k[m]) = \text{rank}_{F_K} (\cB[m])$ for all $k \notin K$ and $m\geq 0$. We have a vertex algebra $\cB^{\infty} = \lim_{\kappa\ra \infty} \cB$ with basis $\{\alpha_i|\ i\in I\}$, where $\{a_i|\ i \in I\}$ is any basis of $\cB$ over $F_K$, and $\alpha_i = \lim_{\kappa \ra \infty} a_i$. By construction, $\text{dim}_{\mathbb{C}}(\cB^{\infty}[m]) = \text{rank}_{F_K}(\cB[m])$ for all $m\geq 0$. The vertex algebra structure on $\cB^{\infty}$ is defined by \begin{equation} (\alpha_i)_{(n)}  (\alpha_j) = \lim_{\kappa \ra \infty} (a_i)_{(n)}  (a_j), \qquad i,j\in I, \qquad n\in \mathbb{Z}. \end{equation} The $F_K$-linear map $\phi: \cB \ra \cB^{\infty}$ sending $a_i \mapsto \alpha_i$ satisfies \begin{equation} \label{preservecircle} \phi(\omega_{(n)}  \nu) = \phi(\omega)_{(n)}  \phi(\nu), \qquad \omega,\nu \in \cB, \qquad n\in \mathbb{Z}.\end{equation} 

Let $U = \{\alpha_i|\ i\in I\}$ be a strong generating set for $\cB^{\infty}$, and let $T = \{a_i|\ i\in I\}$ be the corresponding subset of $\cB$, so that $\phi(a_i) = \alpha_i$. By Lemma 8.1 of \cite{CLI}, the set $T$ closes under OPE. In other words, each term appearing in the OPE of $a_i(z) a_j(w)$ for $i,j\in I$ can be expressed as a linear combination of normally ordered monomials of the form \begin{equation} \label{nop} :\partial^{k_1} a_{i_1}(w) \cdots \partial^{k_r} a_{i_r}(w):,\end{equation} for $i_1,\dots, i_r \in I$ and $k_1,\dots, k_r \geq 0$. The {\it structure constants}, i.e., the coefficients of these monomials, are rational function of $\kappa$ whose poles need not lie in $K$. Let $D$ be the set of all poles of structure constants appearing in the OPE of $a_i(z) a_j(w)$ for $i,j\in I$.

\begin{lemma}[\cite{CLII}, Lemma 3.4] \label{passage} Let $K\subset \mathbb{C}$ be at most countable, and let $\cB$ be a vertex algebra over $F_K$ as above. Let $U = \{\alpha_i|\ i\in I\}$ be a strong generating set for $\cB^{\infty}$, and let $T = \{a_i|\ i\in I\}$ be the corresponding subset of $\cB$, so that $\phi(a_i) = \alpha_i$. Letting $S = K \cup D$, $F_S \otimes_{F_K}\cB$ is strongly generated by $T$. Here we have identified $T$ with the set $\{1 \otimes a_i|\ i\in I\} \subset F_S \otimes_{F_K} \cB$. \end{lemma}

\begin{cor} \label{passagecor} For $k\notin S$, the vertex algebra $\cB^k = \cB/ (\kappa -k)$ is strongly generated by the image of $T$ under the map $\cB \ra \cB^k$.
\end{cor}

If $U$ is a {\it minimal} strong generating set for $\cB^{\infty}$, $T$ need not be a minimal strong generating set for $\cB$, since there may exist relations of the form $\lambda(k) \alpha_j = P$, where $P$ is a normally ordered polynomial in $\{\alpha_i|\ i\in I,\ i\neq k\}$ and $\lim_{k\ra \infty} \lambda(k) = 0$, although $\lim_{k\ra \infty} P$ is a nontrivial. However, there is one condition which holds in many examples, under which $T$ is a minimal strong generating set for $\cB$.

\begin{prop}[\cite{CLII}, Lemma 3.6] Suppose that $U = \{\alpha_i|\ i\in I\}$ is a minimal strong generating set for $\cB^{\infty}$ such that $\text{wt}(\alpha_i) <N$ for all $i\in I$. If there are no normally ordered polynomial relations among $\{\alpha_i|\ i\in I\}$ and their derivatives of weight less than $N$, the corresponding set $T = \{a_i|\ i\in I\}$ is a minimal strong generating set for $\cB$. \end{prop}

In our main example $\cW^k(\gg,e_{-\theta})$, suppose that $$\text{dim}(\gg^{\natural}_0) = n,\qquad \text{dim}(\gg^{\natural}_1) = 2m, \qquad \text{dim}(\gg^{\text{ev}}_{-1/2}) = 2r,\qquad \text{dim}(\gg^{\text{odd}}_{-1/2}) = s.$$ If we replace the generators $J^a$, $L$, and $G^u$ with $\tilde{J}^a = \frac{1}{\sqrt{k}} J^a$, $\tilde{L} = \frac{1}{\sqrt{k}}L$, and $\tilde{G}^u = \frac{1}{k} G^u$, respectively, then it follows from Theorem 5.1 of \cite{KWI} that all structure constants are rational functions of $\sqrt{k}$ of degree at most zero. Moreover, the only structure constants which have degree zero as rational functions of $\sqrt{k}$ are the following.
\begin{enumerate}
\item The second-order pole of $\tilde{J}^a(z) \tilde{J}^b(w)$ when $(a|b) \neq 0$.
\item The third-order pole of $\tilde{G}^u(z) \tilde{G}^v(w)$ when $\langle u,v\rangle_{\text{ne}} \neq 0$.
\item The fourth-order pole of $\tilde{L}(z)\tilde{L}(w)$.
\end{enumerate}

Let $\kappa$ be a formal variable satisfying $\kappa^2 = k$, and let $F = F_K$ for $K =  \{0\}$. Let $\cW$ be the vertex algebra with coefficients in $F$ which is freely generated by $\bar{J}^a, \bar{L}, \bar{G}^u$ satisfying the rescaled OPE relations. It follows that for $k\neq 0$, we have a surjective vertex algebra homomorphism $$\cW \ra \cW^k(\gg, e_{-\theta}),\qquad \bar{J}^a \mapsto \frac{1}{\sqrt{k}} J^a,\qquad \bar{L} \mapsto \frac{1}{\sqrt{k}} L,\qquad \bar{G}^u \mapsto \frac{1}{k} G^u,$$ whose kernel is the ideal $(\kappa - \sqrt{k})$. Then $\cW^k(\gg, e_{-\theta}) \cong \cW/ (\kappa - \sqrt{k})$, and the limit is given by $$\cW^{\infty} = \lim_{\kappa \ra\infty} \cW \cong \cH(n) \otimes \cA(m) \otimes \cT \otimes \cG_{\text{ev}}(r)\otimes \cG_{\text{odd}}(s).$$ In this notation,
\begin{enumerate}
\item $\cH(n)$ is the rank $n$ Heisenberg algebra, which has even generators $\alpha^i$, $i=1,\dots, n$, and operator products
$$\alpha^i(z) \alpha^j(w) \sim \delta_{i,j}(z-w)^{-2}.$$ 
\item $\cA(m)$ is the rank $m$ symplectic fermion algebra, which has odd generators $e^i, f^i$, $i=1,\dots, m$, and operator products 
$$e^i(z) f^j(w) \sim \delta_{i,j}(z-w)^{-2}.$$ 
\item $\cT$ is a generalized free field algebra with even generator $t(z)$ in weight two satisfying $$t(z)t(w) \sim (z-w)^{-4}.$$ 
\item $\cG_{\text{ev}}(r)$ is a generalized free field algebra with even generators $a^i, b^i$, $i=1,\dots, r$, and operator products 
$$a^i(z) b^j(w) \sim \delta_{i,j}(z-w)^{-3}.$$ 
\item $\cG_{\text{odd}}(s)$ is a generalized free field algebra with odd generators $\phi^i$, $i=1,\dots, s$, and operator products 
$$\phi^i(z) \phi^j(w) \sim \delta_{i,j}(z-w)^{-3}.$$
\end{enumerate}

Note that $\cT$, $\cG_{\text{ev}}(r)$, and $\cG_{\text{odd}}(s)$ do not possess conformal vectors. However, they have quasi-conformal structures \cite{FBZ} such that $t$ has weight $2$, and $a^i, b^i, \phi^i$ have weight $3/2$. To see this, note that $\cT$ can be regarded as the subalgebra of the rank one Heisenberg algebra generated by the derivative of the generator. Similarly, $\cG_{\text{ev}}(r)$ can be regarded as the subalgebra of the rank $r$ $\beta\gamma$-system generated by the derivatives of the generators. Finally, $\cG_{\text{odd}}(s)$ can be regarded as the subalgebra of the rank $s$ free fermion algebra generated by the derivatives of the generators. The quasi-conformal structures on $\cT$, $\cG_{\text{ev}}(r)$, and $\cG_{\text{odd}}(s)$ are then inherited from the conformal structures on the Heisenberg algebra, $\beta\gamma$-system, and free fermion algebra, respectively.

The full automorphism groups of $\cG_{\text{ev}}(r)$ and $\cG_{\text{odd}}(s)$ which preserve the weight gradation are the symplectic group $\text{Sp}(2r)$ and the orthogonal group $\text{O}(s)$, respectively. It is necessary to understand the structure of orbifolds of $\cG_{\text{ev}}(r)$ and $\cG_{\text{odd}}(s)$ under arbitrary reductive groups. As in our previous studies of orbifolds of free field algebras, we first need to describe $\cG_{\text{ev}}(r)^{\text{Sp}(2r)}$ and $\cG_{\text{odd}}(s)^{\text{O}(s)}$.

\begin{thm} \label{thm:sp-invariant} $\cG_{\text{ev}}(r)^{\text{Sp}(2r)}$ has a minimal strong generating set $$\omega^{2j+1} = \frac{1}{2} \sum_{i=1}^r \big(:a^i \partial^{2j+1} b^i: \  - \  :(\partial^{2j+1} a^i )b^i:\big),\qquad 0\leq j \leq r^2+3r-1,$$ and is therefore of type $\cW(4,6,\dots, 2r^2+6r+2)$. Moreover, $\cG_{\text{ev}}(r)$ is completely reducible as a $\cG_{\text{ev}}(r)^{\text{Sp}(2r)}$-module, and all irreducible modules in this decomposition are highest-weight and $C_1$-cofinite according to Miyamoto's definition \cite{Mi}.
\end{thm}

\begin{proof} The first statement can be reduced to showing that, in the notation of Equation 9.1 of \cite{LV}, $R_r(I) \neq 0$ for $I = (1,2,\dots,2r+2)$. The explicit formula for $R_r(I)$ is given by Theorem 4 of \cite{LV}, and it is clear that it is nonzero. The Zhu algebra \cite{Z} of $\cG_{\text{ev}}(r)^{\text{Sp}(2r)}$ is abelian, which implies that its irreducible, positive energy representations are all highest-weight. The proof of $C_1$-cofiniteness is the same as the proof of Lemma 8 of \cite{LII}. \end{proof}

\begin{cor} For any reductive group $G\subset \text{Sp}(2r)$, $\cG_{\text{ev}}(r)^G$ is strongly finitely generated.
\end{cor}

\begin{proof} This is the same as the proof of Theorem 15 of \cite{LV}. First, $\cG_{\text{ev}}(r)^G$ is completely reducible as a $\cG_{\text{ev}}(r)^{\text{Sp}(2r)}$-module. By a classical theorem of Weyl (Theorem 2.5A of \cite{W}; see also \cite{GW} for a modern treatment), it has an (infinite) strong generating set that lies in the sum of finitely many irreducible $\cG_{\text{ev}}(r)^{\text{Sp}(2r)}$-modules. The result then follows from the strong finite generation of $\cG_{\text{ev}}(r)^{\text{Sp}(2r)}$ and the $C_1$-cofiniteness of these modules.
\end{proof}

There is one case where an explicit description of $\cG_{\text{ev}}(r)^G$ will be important to us, namely, $G = \text{GL}(r) \subset \text{Sp}(2r)$, where $\{a^i\}$ spans a copy of the standard module $\mathbb{C}^r$ and $\{b^i\}$ spans a copy of $(\mathbb{C}^r)^*$.

\begin{thm} \label{thm:gl-invariant} $\cG_{\text{ev}}(r)^{\text{GL}(r)}$ has a minimal strong generating set $$\nu^{j} =\sum_{i=1}^r :a^i \partial^{j} b^i:,\qquad 0\leq j \leq r^2+4r-1,$$ and is therefore of type $\cW(3,4,\dots, r^2+4r+2)$. Moreover, $\cG_{\text{ev}}(r)$ is completely reducible as a $\cG_{\text{ev}}(r)^{\text{GL}(r)}$-module, and all irreducible modules in this decomposition are highest-weight and $C_1$-cofinite according to Miyamoto's definition.
\end{thm}

\begin{proof} The method of \cite{LI} for studying the $\cW_{1+\infty,-r}$-algebra via its realization as the $\text{GL}(r)$-invariants in the rank $r$ $\beta\gamma$-system can be applied in this case.
\end{proof}

\begin{thm} \label{thm:o-invariant} $\cG_{\text{odd}}(s)^{\text{O}(s)}$ has a minimal strong generating set $$\omega^{2j+1} = \frac{1}{2} \sum_{i=1}^s :\phi^i \partial^{2j+1} \phi^i:,\qquad 0\leq j \leq 2s-1,$$ and is therefore of type $\cW(4,6,\dots, 4s+2)$. Moreover, $\cG_{\text{odd}}(s)$ is completely reducible as a $\cG_{\text{odd}}(s)^{\text{O}(s)}$-module, and all irreducible modules in this decomposition are highest-weight and $C_1$-cofinite according to Miyamoto's definition.
\end{thm}

\begin{proof} This can be reduced to showing that, in the notation of Equation 11.1 of \cite{LV}, $R_s(I,J) \neq 0$ for $I = (1,1,\dots, 1)$ and $J = (2,2,\dots, 2)$, where both lists have length $s+1$. This follows easily from the recursive formula given by Equation 11.5 of \cite{LV}.
\end{proof}

\begin{cor} For any reductive group $G\subset \text{O}(s)$, $\cG_{\text{odd}}(s)^G$ is strongly finitely generated.
\end{cor}

Recall that the full automorphism groups of $\cH(n)$ and $\cA(m)$ are $\text{O}(n)$ and $\text{Sp}(2m)$, respectively. Suppose that $G$ is a reductive group of automorphisms of $$\cH(n) \otimes \cA(m) \otimes \cG_{\text{ev}}(r) \otimes \cG_{\text{odd}}(s)$$ which preserves the tensor factors, so that $$G \subset  \text{O}(n) \times \text{Sp}(2m) \times  \text{Sp}(2r) \times  \text{O}(s).$$ By the same argument as the proof of Theorem 4.2 of \cite{CLII}, we obtain

\begin{cor} \label{cororbiflod} $(\cH(n) \otimes \cA(m) \otimes \cG_{\text{ev}}(r) \otimes \cG_{\text{odd}}(s))^G$ is strongly finitely generated.
\end{cor}

Let $G\subset G^{\natural}_0$ be a reductive group of automorphisms of $\cW^k(\gg, e_{-\theta})$. Then $G$ acts on the deformable family $\cW$ over $F$ above, satisfying $\cW / (\kappa - \sqrt{k}) \cong \cW^k(\gg, e_{-\theta})$ for $k\neq 0$, and $\cW^G$ is also a deformable family satisfying $\cW^G / (\kappa - \sqrt{k}) \cong \cW^k(\gg, e_{-\theta})^G$ for $k\neq 0$. Moreover, we have
\begin{equation} \label{orbwlimit} \lim_{\kappa \ra \infty} \cW^G  \cong  \cT \otimes (\cH(n) \otimes \cA(m) \otimes \cG_{\text{ev}}(r) \otimes \cG_{\text{odd}}(s))^G.\end{equation} The argument is the same as the proof of Lemma 5.1 and Corollary 5.2 of \cite{CLII}. 

\begin{thm} For any reductive group $G$ of automorphisms of $\cW^k(\gg, e_{-\theta})$, $\cW^k(\gg, e_{-\theta})^G$ is strongly finitely generated for all but finitely many values of $k$.
\end{thm}

\begin{proof}
Since $G$ preserves the tensor factors of $\cT \otimes \cH(n) \otimes \cA(m) \otimes \cG_{\text{ev}}(r) \otimes \cG_{\text{odd}}(s)$, the right hand side of \eqref{orbwlimit} is strongly generated by some finite set $U  = \{\alpha_1,\dots, \alpha_r\}$. It follows from Theorem \ref{passage} and Corollaries \ref{passagecor} and \ref{cororbiflod} that the corresponding set $T =\{a_i,\dots, a_r\} \subset \cW^G$ such that $\lim_{\kappa\ra \infty} a_i = \alpha_i$, descends to a strong generating set for $\cW^k(\gg, e_{-\theta})^G$ for all $k\notin \{0\} \cup D$. Here $D$ is the set of poles of structure constants appearing in the OPE of $a_i(z) a_j(w)$ for $i,j = 1,\dots, r$.
\end{proof}

Next, suppose that $\gg' \subset \gg^{\natural}$ is a simple Lie subalgebra. Let $V^{\ell}(\gg')$ be the corresponding affine vertex subalgebra of $\cW^k(\gg, e_{-\theta})$, and observe that $\cW^k(\gg, e_{-\theta})$ is {\it good} in the sense of Definition 6.1 of \cite{CLII}. Let $\cV$ be the deformable family over $F$ such that $\cV / (\kappa - \sqrt{\ell}) \cong V^{\ell}(\gg')$. We have a homomorphism of deformable families $\cV \ra \cW$, and we define $$\cC = \text{Com}(\cV, \cW).$$

\begin{thm} \label{propertiesofc}
\begin{enumerate} 
\item $\cC$ is a deformable family over $F_{K}$, where $K$ is a subset of $\mathbb{C}$ containing $0$ together with a set of rational numbers $k$ such that $\ell\leq -h^{\vee}$. Here $h^{\vee}$ is the dual Coxeter number of $\gg'$. 
\item For all $k\notin K$, $\cC^k = \cC/ (\kappa - \sqrt{k})$ coincides with the coset $\text{Com}(V^{\ell}(\gg'), \cW^k(\gg, e_{-\theta}))$. \end{enumerate} \end{thm}

\begin{proof} This follows from Corollary 6.7 of \cite{CLII}.
\end{proof}

More generally, suppose that $\gg'\subset \gg^{\natural}$ is {\it reductive}, i.e., a sum of abelian and simple ideals, and that the restriction of the bilinear form on $\gg^{\natural}$ to each ideal of $\gg'$ is nondegenerate. Then the affine vertex algebra of $\gg'$ will be the tensor product of a Heisenberg algebra and $V^{\ell_i}(\gg_i)$ for the various simple ideals $\gg_i \subset \gg'$. We denote this affine vertex algebra by $V(\gg')$. By Remark 6.9 of \cite{CLII}, the statement of Theorem \ref{propertiesofc} must be modified slightly: $\cC$ is a deformable family over $F_K$, where $K$ contains $0$ and a set of rational numbers $k$ such that $\ell_i  \leq - h_i^{\vee}$, for some $i$. Here $h_i^{\vee}$ is the dual Coxeter number of $\gg_i$. Then $\cC^k = \text{Com}(V(\gg'), \cW^k(\gg, e_{-\theta}))$ for all $k\notin K$.

Let $\gg'$ be reductive as above, and let $d = \text{dim}(\gg')$. Observe that  $\lim_{\kappa \ra \infty} \cV \cong \cH(d)$ and $$\lim_{\kappa \ra \infty} \cW \cong \cH(d) \otimes \tilde{\cW},\qquad \tilde{\cW} =  \cT \otimes \cH(n-d) \otimes  \cA(m) \otimes \cG_{\text{ev}}(r) \otimes \cG_{\text{odd}}(s).$$ 
By Theorem 6.8 of \cite{CLII}, there is a connected Lie group $G$ with Lie algebra $\gg'$ such that $G$ acts on $\tilde{\cW}$, and
$$\lim_{\kappa \ra \infty} \cC  \cong \tilde{\cW}^G =  \cT \otimes (\cH(n-d) \otimes \cA(m) \otimes \cG_{\text{ev}}(r) \otimes \cG_{\text{odd}}(s))^G.$$ Also $G$ preserves the tensor factors, so $\cT \otimes (\cH(n-d) \otimes \cA(m) \otimes \cG_{\text{ev}}(r) \otimes \cG_{\text{odd}}(s))^G$ has a strong finite generating set $U = \{\alpha_1,\dots, \alpha_r\}$. Let $T= \{a_1,\dots, a_r\}$ be the corresponding subset of $\cC$ such that $\lim_{\kappa\ra \infty} a_i = \alpha_i$.

\begin{thm} \label{mainthm:cosetkd} For $\gg$, $\gg'$, $K$, and $T = \{a_1,\dots, a_r\}\subset \cC$ as above, $T$ descends to a strong generating set for $\text{Com}(V(\gg'), \cW^k(\gg, e_{-\theta}))$ for all $k \notin K \cup D$. Here $D$ is the set of poles of the structure constants appearing in the OPE of $a_i(z) a_j(w)$ for $i,j = 1,\dots, r$. \end{thm}
 
\begin{proof} For $k \notin K$, $\cC^k = \text{Com}(V(\gg'), \cW^k(\gg, e_{-\theta}))$, so this follows immediately from Lemma \ref{passage} and Corollary \ref{passagecor}.\end{proof}

\begin{remark} \label{mainthm:remark} It is often possible to choose an infinite strong generating set $U' = \{\alpha_i|\ i\in I\}$ for $\tilde{\cW}^G$ with the property that the set $D'$ of poles of the structure constants in the corresponding set $T' = \{a_i|\ i\in I\}$ is easier to describe than the set $D$ above. By the same argument as Theorem \ref{mainthm:cosetkd}, $T'$ strongly generates $\text{Com}(V(\gg'), \cW^k(\gg, e_{-\theta}))$ for all $k \notin K \cup D'$. As we shall see in some examples, there are values of $k$ which do not lie in $D'$ but cannot easily be shown not to lie in $D$. This gives us a way to find a strong generating set for $\text{Com}(V(\gg'), \cW^k(\gg, e_{-\theta}))$ at these values. \end{remark}

By a slight abuse of notation for the rest of this paper, if $\cW$ is the deformable family such that $\cW / (\kappa - \sqrt{k}) \cong \cW^k(\gg, e_{-\theta})$, we write $\lim_{k\ra \infty} \cW^k(\gg, e_{-\theta})$ instead of $\lim_{\kappa \ra \infty} \cW$. Similarly, if $\cC$ is the deformable family such that $\cC/ (\kappa - \sqrt{k}) \cong \text{Com}(V(\gg'), \cW^k(\gg, e_{-\theta}))$ for generic $k$, we write $\lim_{k\ra \infty} \cC^k$ instead of $\lim_{\kappa \ra \infty} \cC$.

\section{Generic structure of $\text{Com}(V^{k'}(\gg^{\natural}), \cW^k(\gg, e_{-\theta}))$: some explicit results}\label{sect:genericbehavior}
We call $k\in \mathbb{C}$ {\it generic} if $k \notin K \cup D$, where $K$ and $D$ are as in Theorems \ref{propertiesofc} and \ref{mainthm:cosetkd}. In this section, we shall find minimal strong generating sets for $\text{Com}(V^{k'}(\gg^{\natural}), \cW^k(\gg, e_{-\theta}))$ for generic values of $k$ in a few cases.

\subsection{The case $\gg = \gs\gl_n$} Recall from \cite{KWI} that $(\gs\gl_n)^{\natural} = \gg\gl_{n-2}$, and $$(\gs\gl_n)_{-1/2} \cong \mathbb{C}^{n-2} \oplus (\mathbb{C}^{n-2})^*$$ as $\gg\gl_{n-2}$-modules. The weight one subspace generates a copy of $V^{k+1}(\gg\gl_{n-2}) = \cH \otimes V^{k+1}(\gs\gl_{n-2})$, where $\cH$ is a rank one Heisenberg algebra. It follows that $$\lim_{k\ra \infty} \cW^k(\gs\gl_n, e_{-\theta}) \cong \cH((n-2)^2) \otimes \cT \otimes \cG_{\text{ev}}(n-2),$$ and $$\lim_{k\ra \infty} \text{Com}(V^{k+1}(\gg\gl_{n-2}),\cW^k(\gs\gl_n, e_{-\theta})) \cong \cT \otimes \cG_{\text{ev}}(n-2)^{\text{GL}(n-2)}.$$ By Theorem \ref{thm:gl-invariant}, $\cG_{\text{ev}}(n-2)^{\text{GL}(n-2)}$ is of type $\cW(3,4,\dots, n^2-2)$, so we obtain

\begin{thm} \label{thm:slncoset}$\text{Com}(V^{k+1}(\gg\gl_{n-2}),\cW^k(\gs\gl_n, e_{-\theta}))$ is of type $\cW(2,3,\dots, n^2-2)$ for generic values of $k$.
\end{thm}

In the case $n=3$, this generalizes the result of \cite{ACL} that $\text{Com}(\cH, \cW(\gs\gl_3, e_{-\theta}))$ is generically of type $\cW(2,3,4,5,6,7)$.

\subsection{The case $\gg = \gs\gp_{2n}$} Recall from \cite{KWI} that for $n\geq 2$, $(\gs\gp_{2n})^{\natural} \cong \gs\gp_{2n-2}$ and that $$(\gs\gp_{2n})_{-1/2} \cong \mathbb{C}^{2n-2}$$ as $\gs\gp_{2n-2}$-modules. The weight one subspace generates a copy of $V^{k+1/2}(\gs\gp_{2n-2})$. It follows that $$\lim_{k\ra \infty} \cW^k(\gs\gp_{2n}, e_{-\theta}) \cong \cH(d) \otimes \cT \otimes \cG_{\text{ev}}(n-2),\qquad d = \text{dim}(\gs\gl_{2n-2}),$$ and that $$\lim_{k\ra \infty} \text{Com}(V^{k+1/2}(\gs\gp_{2n-2}),\cW^k(\gs\gp_{2n}, e_{-\theta})) \cong \cT \otimes \cG_{\text{ev}}(n-2)^{\text{Sp}(2n-2)}.$$ By Theorem \ref{thm:sp-invariant},  $\cG_{\text{ev}}(n-2)^{\text{Sp}(2n-2)}$ is of type $\cW(4,6,\dots, 2n^2+2n-2)$. We obtain

\begin{thm} \label{generic:typeC} $\text{Com}(V^{k+1/2}(\gs\gp_{2n-2}),\cW^k(\gs\gp_{2n}, e_{-\theta}))$ is of type $\cW(2,4,\dots, 2n^2+2n-2)$ for generic values of $k$.
\end{thm}

\subsection{The case $\gg = \gs\gl(2|n)$, $n\neq 2$} In this case, $\gg^{\natural} \cong \gg\gl_n$ and $$\gg_{-1/2} \cong \mathbb{C}^n \oplus (\mathbb{C}^n)^*,$$ regarded as an odd vector space. The weight one subspace generates a copy of $V^{k-1}(\gg\gl_{n}) = \cH \otimes V^{k-1}(\gs\gl_{n})$, and $$\lim_{k\ra \infty} \cW^k(\gs\gl(2|n), e_{-\theta}) \cong \cH(n^2) \otimes \cT \otimes \cG_{\text{odd}}(2n).$$
Letting $\cC^k = \text{Com}(V^{k-1}(\gg\gl_n), \cW^k(\gs\gl(2|n), e_{-\theta}))$, we have $$\lim_{k \ra \infty} \cC^k  \cong \cT \otimes \cG_{\text{odd}}(2n)^{\text{GL}(n)}.$$ By the same argument as Theorem 4.3 of \cite{CLIII}, it can be shown that $\cG_{\text{odd}}(2n)^{\text{GL}(n)}$ is purely even and of type $\cW(3,4,\dots, 3n+2)$. We obtain

\begin{thm} For $n\neq 2$, $\text{Com}(V^{k-1}(\gg\gl_n), \cW^k(\gs\gl(2|n), e_{-\theta}))$ is even and is generically of type $\cW(2,3,\dots, 3n+2)$.
\end{thm}

This generalizes the case $n=1$, where $\cW^k(\gs\gl(2|1), e_{-\theta})$ is just the $N=2$ superconformal vertex algebra. In this case, the above commutant is known to be of type $\cW(2,3,4,5)$ and is isomorphic to the universal parafermion algebra of $\gs\gl_2$.

\subsection{The case $\gg = \gp \gs\gl(2|2)$} Recall that in this case, $\cW^k(\gp \gs\gl(2|2), e_{-\theta})$ is isomorphic to the small $N=4$ superconformal vertex algebra. We have $\gg^{\natural} = \gs\gl_2$ and $$\gg_{-1/2} \cong \mathbb{C}^2 \oplus (\mathbb{C}^2)^*,$$ regarded as an odd vector space. The weight one subspace generates a copy of $V^{k-1}(\gs\gl_2)$. Letting $\cC^k = \text{Com}(V^{k-1}(\gs\gl_2), \cW^k(\gp \gs\gl(2|2), e_{-\theta}))$, we have $$\lim_{k\ra \infty} \cC^k \cong \cT \otimes \cG_{\text{odd}}(4)^{\text{SL}(2)}.$$ It is not difficult to show that $\cG_{\text{odd}}(4)^{\text{SL}(2)}$ is even and of type $\cW(3^3,4, 5^3, 6, 7^3, 8)$. In other words, a minimal strong generating set consists of three even fields in each weight $3, 5, 7$, and one even field in weights $4,6,8$. We obtain

\begin{thm} For $\gg = \gp\gs\gl(2|2)$, $\text{Com}(V^{k-1}(\gs\gl_2), \cW^k( \gp \gs\gl(2|2), e_{-\theta}))$ is even and is generically of type $\cW(2,3^3,4, 5^3, 6, 7^3, 8)$. \end{thm}

\subsection{The case $\gg = \go\gs\gp(1|4)$} In all the above examples, $\gg^{\natural}$ is a Lie algebra. The case $\gg = \go\gs\gp(1|4)$ is an example where $\gg^{\natural}$ is the Lie superalgebra $\go \gs\gp(1|2)$, but we can still describe $$\cC^k = \text{Com}(V^{k+1/2}(\go\gs\gp(1|2)),  \cW^k(\go\gs\gp(1|4), e_{-\theta})).$$ 
We have $\gg_{-1/2} \cong \mathbb{C}^{2|1}$, i.e., the superspace whose even part has dimension $2$ and odd part has dimension $1$. Since $\text{Osp}(1|2)$ acts completely reducibly on $\cW^k(\go\gs\gp(1|4), e_{-\theta})$, it can be shown using the methods of \cite{CLII} that $$\lim_{k\ra \infty} \cC^k \cong \cT \otimes  \big(\cG_{\text{ev}}(1) \otimes \cG_{\text{odd}}(1)\big)^{\text{Osp}(1|2)}.$$ 
By the same argument as Theorem 12 of \cite{LV}, which is based on Sergeev's first and second fundamental theorems of invariant theory for $\text{Osp}(1|2)$ \cite{SI,SII}, it can be shown that $\big(\cG_{\text{ev}}(1) \otimes \cG_{\text{odd}}(1)\big)^{\text{Osp}(1|2)}$ is purely even and of type $\cW(4,6,\dots,18)$. We obtain

\begin{thm} For $\gg = \go\gs\gp(1|4)$, $\text{Com}(V^{k+1/2}(\go\gs\gp(1|2)), \cW^k(\go\gs\gp(1|4), e_{-\theta}))$ is even and is generically of type $\cW(2,4,\dots, 18)$.
\end{thm}

\section{The structure of the nongeneric set}\label{sect:nongeneric}
In the notation of Theorems \ref{propertiesofc} and \ref{mainthm:cosetkd}, the set of {\it nongeneric} values of $k$, where $T = \{a_1\dots, a_r\}$ can fail to strongly generate $\text{Com}(V^{k'}(\gg^{\natural}), \cW^k(\gg, e_{-\theta}))$, is a subset of $K \cup D$. Unfortunately, this is not explicit enough for our purposes, the goal being to describe this coset at special values of $k$ for which it is highly reducible. 

In \cite{ACL}, it was shown that in the case $\gg = \gs\gl_3$, $\cC^k = \text{Com}(\cH, \cW^k(\gs\gl_3, e_{-\theta}))$ is of type $\cW(2,3,4,5,6,7)$ for all values of $k$ except for $-1$ and $-3/2$. This is easier than the general situation because cosets of Heisenberg algebras are better behaved than cosets of general affine vertex algebras. In this section, we analyze the structure of the nongeneric set for $\cC^k$ in two examples where $\gg^{\natural}$ is nonabelian in order to illustrate how this problem can be handled computationally. As suggested by Remark \ref{mainthm:remark}, our approach is to first find an {\it infinite} strong generating set with the property that we can determine the poles of the structure constants appearing in the OPEs of the generators.

\subsection{The case $\gg = \gs\gp_{4}$}

The generators of $\cW^k(\gs\gp_{4}, e_{-\theta})$ are $T, H,X,Y,G^{\pm} $, where $X,Y,H$ generate a copy of $V^{k+1/2}(\gs\gl_2)$, $T$ is a Virasoro of central charge $$c=-\frac{3 (1 + k) (1 + 2 k)}{(3 + k)},$$ and $G^{\pm}$ are primary of weight $3/2$ and satisfy
$$H(x) G^{\pm}(w) \sim  G^{\pm}(w)(z-w)^{-1},\qquad X(z) G^-(w) \sim G^+(w)(z-w)^{-1},$$ $$Y(z) G^+(w) \sim  G^-(w)(z-w)^{-1}, $$ $$G^+(z) G^+(w) \sim (8 + 4 k) X(w) (z-w)^{-2} +  (4 + 2 k) \partial X(w) (z-w)^{-1},$$
$$G^-(z) G^-(w) \sim - (8 +4 k) Y(w) (z-w)^{-2} - (4 + 2 k) Y(w) (z-w)^{-1},$$
$$G^+(z) G^-(w) \sim -(4 +10 k + 4 k^2) (z-w)^{-3}  - (4 +2 k) H(w) (z-w)^{-2} $$ $$+  \big((6 + 2 k) T- 4 :XY:  - :HH: - k \partial H\big)(w) (z-w)^{-1}.$$
We use the notation $a = \lim_{k \ra \infty} \frac{1}{k} G^+$ and $b = \lim_{k \ra \infty} \frac{1}{k} G^-$ for the generators of $\cG_{\text{ev}}(1)$, which satisfy
$$a(z) b(w) \sim -4 (z-w)^{-3}.$$ Note that this normalization is different from the one used earlier. Recall that by classical invariant theory, the orbifold $\cG_{\text{ev}}(1)^{Sp(2)}$ has generators 
$$\omega_{i,j} = \ :\partial^i a \partial^j b:\  - \ :\partial^j a \partial^i b:,\qquad 0\leq i < j,$$ and that the smaller set $\{\omega_{0,n}|\ n=1,3,5,\dots\}$ suffices. 

By Theorem \ref{generic:typeC}, $\cC^k(2) = \text{Com}(V^{k+1/2}(\gs\gp_{2}), \cW^k(\gs\gp_{4},e_{-\theta}))$ is of type $\cW(2,4,6,8,10)$ for generic $k$, and has Virasoro element 
$$L = T - \frac{1}{10+4k} \big(:HH: + 4 :XY: -2 \partial H \big)$$ with central charge $$c = -\frac{6 (2 + k)^2 (1 + 2 k)}{(3 + k) (5 + 2 k)}.$$ In order to write down additional elements of $\cC^k(2)$, we first consider
$$u_{0,n} = \ :G^+ \partial^n G^-:\  - \ :(\partial^j G^+) G^-: \ \in \cW^k(\gs\gp_{4}, e_{-\theta}),\qquad n = 1,3,5,\dots.$$ Clearly $u_{0,n}$ has weight $n+3$ and $$\lim_{k\ra \infty} \frac{1}{k^2} u_{0,n} = \omega_{0,n},$$ but $u_{0,n} \notin \cC^k$. We shall construct suitable corrections 
\begin{equation} \label{corrections} U_{0,n} = u_{0,n} + P_n \in \cC^k,\end{equation}
 such that $P_n$ is a normally ordered polynomial in $T,H,X,Y,u_{0,1},\dots u_{0,n-2}$ and their derivatives, and $\lim_{k \ra \infty} \frac{1}{k^2} P_n = 0$. First, for $k\neq -9/2$, we define
$$U_{0,1} = u_{0,1} + \frac{2 (2 + k) (3 + k)}{9 + 2 k} \partial^2 T + \frac{13 + 8 k + k^2}{3 (9 + 2 k)} \partial^3 H $$ $$ -  \frac{2}{9 + 2 k}  :(\partial^2 H) H:  - \frac{2 + k}{2 (9 + 2 k)} :(\partial H) \partial H: - \frac{2 (2 + k)}{9 + 2 k} :(\partial X) \partial Y : $$ $$ + \frac{2}{9 + 2 k} :T H H:  + 
\frac{8}{9 + 2 k} :TXY:  - \frac{4}{9 + 2 k}  :T \partial H: + 
\frac{4}{9 + 2 k} :H (\partial X) Y:  $$ $$+ 
\frac{4}{9 + 2 k} :H X \partial Y: -  \frac{2 (8 + 3 k)}{9 + 2 k} :TT:  +  \frac{2}{9 + 2 k} : (\partial H) H H:  $$ $$-  \frac{2 (3 + k)}{9 + 2 k} :(\partial T) H: + 
\frac{2}{9 + 2 k} : H G^+ G^-:  +  \frac{2}{9 + 2 k} :X G^-G^-: - \frac{2}{9 + 2 k} : Y G^+G^+: $$
  
This lies in $\cC^k(2)$ but is not primary with respect to $L$. Instead, it satisfies
\begin{equation} \label{LU01} \begin{split} L(z) U_{0,1}(w) \sim -\frac{12 (2 + k)^2 (1 + 2 k) (18 + 7 k)}{(3 + k) (9 + 2 k)} (z-w)^{-6} + \frac{4 (5 + 2 k) (3 + 17 k + 7 k^2)}{(3 + k) (9 + 2 k)} L(w)(z-w)^{-4} \\ + \frac{2 (5 + 2 k) (18 + 7 k)}{9 + 2 k} L(w)(z-w)^{-3} + 4U_{0,1}(w)(z-w)^{-2} + \partial U_{0,1}(w)(z-w)^{-1}.\end{split}\end{equation} Whenever $k$ is not a root of $p(x) = 30 x^3 + 113 x^2+ 59 x -105$, $U_{0,1}$ can be corrected uniquely to a primary element by adding
$$ \frac{2 (5 + 2 k)^2 (-147 - 32 k + 14 k^2)}{(9 + 2 k) (-105 + 59 k + 113 k^2 + 30 k^3)} :LL: - \frac{(3 + k) (5 + 2 k) (-273 + 22 k + 157 k^2 + 42 k^3)}{(9 + 2 k) (-105 + 59 k + 113 k^2 + 30 k^3)} \partial^2 L,$$ but it is easier for our purposes to work with $U_{0,1}$.

\begin{lemma} For all $n = 1, 3, 5, \dots$, we have 
\begin{equation} (U_{0,1})_{(1)}  (u_{0,n}) = f(k,n)  u_{0,n+2} + P_n,\end{equation} where \begin{equation}  f(k,n) = - \frac{2 (3 + k) (5 + n) (7 + 2 k + 5 n + 2 k n) (14 + 4 k + 5 n + 2 k n)}{(9 + 2 k) (1 + n) (2 + n)},\end{equation}  and $P_{n}$ is a normally ordered polynomial in $T,X,Y,H, u_{0,1}, u_{0,3},\dots, u_{0,n}$, and their derivatives. 
\end{lemma}

\begin{proof} The argument is similar to the proof of Theorem 6.3 of \cite{AL}, and is omitted. 
\end{proof}

Let $$S = \{ -\frac{9}{2}, -3\} \cup  \{- \frac{7 + 5 n}{2 (1 + n)}|\ n\geq 1\} \cup \{- \frac{14 + 5 n}{2 (2 + n)}|\ n\geq 1\},$$ which is the set of values of $k$ for which $f(n,k) = 0$ or is undefined.

We would like the corrections $U_{0,n}$ in \eqref{corrections} to satisfy $$(U_{0,1})_{(1)}  (U_{0,n}) = f(k,n) U_{0,n+2} + Q_n,$$ where $Q_n$ is a normally ordered polynomial in  $L, U_{0,1}, U_{0,3},\dots, U_{0,n}$ and their derivatives, and this can be achieved by a bootstrap procedure as we now demonstrate.

One can check by computer that $$(U_{0,1})_{(1)}  (U_{0,1}) = f(k,1) u_{0,3} + P_1,$$ where $P_1$ depends on $T,H,X,Y, u_{0,1}$ and their derivatives. As a first attempt to correct $u_{0,3}$ to an element of $\cC^k(2)$, define $$\tilde{U}_{0,3} = u_{0,3}  + \frac{1}{f(k,1)} P_1,$$ which certainly lies in $\cC^k(2)$ and satisfies $\lim_{k\ra \infty} \frac{1}{k^2} \tilde{U}_{0,3} = \omega_{0,3}$. However, $\tilde{U}_{0,3}$ is not the desired correction since
$$(U_{0,1})_{(1)}  (\tilde{U}_{0,3}) \neq f(k,3)  u_{0,5} + P_3$$ for any normally ordered polynomial $P_3$ in $T,H,X,Y, u_{0,1}, u_{0,3}$ and their derivatives. The reason is that the terms $:T(\partial G^+)G^-:$ and $:TG^+ (\partial G^-):$ appear in $P_1$ with nonzero coefficient, and if we express $(U_{0,1})_{(1)}  (:T(\partial G^+)G^-:)$ and $(U_{0,1})_{(1)}  (:TG^+ (\partial G^-):)$ as normally ordered polynomials in $T,H,X,Y, u_{0,1}, u_{0,3}, u_{0,5}$ and their derivatives, the coefficient of $u_{0,5}$ is nonzero. It is easy to check that the only normally ordered monomials $\nu$ in $T,H,X,Y, u_{0,1}$ and their derivatives which have the property that $(U_{0,1})_{(1)}  \nu$ has a nontrivial coefficient of $u_{0,5}$, are $:T(\partial G^+)G^-:$ and $:TG^+ (\partial G^-):$.  Moreover, these terms can be eliminated as follows. We define $$U_{0,3} = \tilde{U}_{0,3} - \frac{6}{19 + 6 k} :L U_{0,1}:.$$ Then the terms $:T(\partial G^+)G^-:$ and $:TG^+ (\partial G^-):$ do not appear in $U_{0,3}$, and
\begin{equation} \label{u01u03} (U_{0,1})_{(1)}  (U_{0,3}) = f(k,3) u_{0,5} + P_3, \end{equation} where $P_3$ is a normally ordered polynomial in $T,H,X,Y, u_{0,1}, u_{0,3}$ and their derivatives. Also, note that $$(U_{0,1})_{(1)} (U_{0,1}) = f(k,1) U_{0,3}  + \frac{6 f(k,1) }{19 + 6 k} :L U_{0,1}:.$$

As above, we attempt to correct $u_{0,5}$ by defining $$\tilde{U}_{0,5} = u_{0,5} + \frac{1}{f(k,3)} P_3,$$ which lies in $\cC^k(2)$ and satisfies $\lim_{k\ra \infty} \frac{1}{k^2} \tilde{U}_{0,5} = \omega_{0,5}$. However, $$(U_{0,1})_{(1)} ( \tilde{U}_{0,5}) \neq f(k,5) u_{0,7}+ P_5$$ for any $P_5$ depending only on $T,H,X,Y,u_{0,1}, u_{0,3}, u_{0,5}$ and their derivatives. As before, the reason is that $P_3$ contains terms $\nu$ of the form \begin{equation} \label{badterms} :TT (\partial G^+)G^-:,\qquad :TT G^+(\partial G^-):,\qquad :(\partial^i T)(\partial^j G^+)(\partial^{\ell} G^-):,\qquad i+j+\ell =3,\end{equation} which all have the property that the coefficient of $u_{0,7}$ in $(U_{0,1})_{(1)}  \nu$ is nonzero. However, we can correct $\tilde{U}_{0,5}$ as follows:
$$U_{0,5} = \tilde{U}_{0,5}  - \frac{30 (-1 + 5 k + 2 k^2)}{(3 + k) (11 + 4 k) (19 + 6 k) (29 + 10 k)} :L L U_{0,1}: $$ $$ - \frac{5 (42 + 21 k + 2 k^2)}{(3 + k) (11 + 4 k) (29 + 10 k)} :L U_{0,3}:  $$ $$ - \frac{5 (12234 + 13450 k + 5169 k^2 + 788 k^3 + 36 k^4)}{2 (3 + k)^2 (11 + 4 k) (19 + 6 k) (29 + 10 k)} :(\partial^2 L) U_{0,1}:  $$ $$
- \frac{5 (9523 + 8130 k + 2186 k^2 + 180 k^3)}{4 (3 + k) (11 + 4 k) (19 + 6 k) (29 + 10 k)}  :(\partial L) \partial U_{0,1}: $$ $$ - \frac{5 (5 + 2 k) (491 + 256 k + 30 k^2)}{2 (3 + k) (11 + 4 k) (19 + 6 k) (29 + 10 k)} :L \partial^2 U_{0,1}:.$$
Then $U_{0,5}$ has no terms of the form \eqref{badterms}, and $$(U_{0,1})_{(1)} ( U_{0,5}) = f(k,5) u_{0,7} + P_5,$$ where $P_5$ is a normally ordered polynomial in $T,X,Y, u_{0,1}, u_{0,3}, u_{0,5}$ and their derivatives. Also, it is immediate from the above formula for $U_{0,5}$ that
$$(U_{0,1})_{(1)} ( U_{0,3}) = f(k,3) U_{0,5} + Q_3,$$ where $Q_3$ is a linear combination of $$:LLU_{0,1}:,\qquad :L U_{0,3}:,\qquad :(\partial^2 L) U_{0,1}:,\qquad :(\partial L) (\partial U_{0,1}):,\qquad : L (\partial^2 U_{0,1}):.$$ Finally, note that the only values of $k$ where the denominators of any terms appearing in $U_{0,3}$ and $U_{0,5}$ vanish, are elements of $S$. This is a consequence of \eqref{LU01} together with the fact that $U_{0,3}$ and $U_{0,5}$ lie in the algebra generated by $L$ and $U_{0,1}$.
 
More generally, we can continue this process and construct elements $U_{0,n}$ for all $n = 1,3,5,\dots$, with the property that $$(U_{0,1})_{(1)} (U_{0,n}) = f(k,n) U_{0,n+2} + Q_n$$ where $Q_n$ is a linear combination of elements of the form $$:(\partial^{i_1} L) \cdots (\partial^{i_r} L)(\partial^k U_{0,m}):,\qquad m = 1,3,\dots, n,\qquad i_1, \dots, i_r, k \geq 0.$$ 
Since $U_{0,n}$ all lie in the algebra generated by $L$ and $U_{0,1}$, all values of $k$ where the denominator of any term appearing in $U_{0,n}$ vanishes, lie in $S$. Moreover, it is not difficult to check by computer that all structure constants in the OPE of $U_{0,1}(z) U_{0,1}(w)$ have poles lying in $S$. It follows that $$\{L, U_{0,n}|\ n = 1,3,5, \dots\}$$ closes under OPE for all $k\notin S$. In particular, all poles of the structure constants in the OPEs of $L(z) U_{0,m}(w)$ and $U_{0,n}(z) U_{0,m}(w)$ lie in $S$, for all $n,m = 1,3,5,\dots$.

\begin{cor} For all real numbers $k>-5/2$, $\cC^k(2)$ is strongly generated by $$\{L, U_{0,n}|\ n = 1,3,5,\dots \}.$$ 
\end{cor}
\begin{proof} Since $\cC^k(2) = \cC / (\kappa - \sqrt{k})$ for all real numbers $k>-5/2$, and $S$ does not intersect this set, this is immediate from Remark \ref{mainthm:remark}. \end{proof}
 
Next, we consider normally ordered relations among the generators. Recall that the first normally ordered relation among generators $\omega_{i,j}$ of $\cG_{\text{ev}}(1)^{Sp(2)}$ occurs at weight $12$, and is a quantum correction of a classical Pfaffian relation. The leading term is
$$:\omega_{0,1} \omega_{2,3}: - :\omega_{0,2} \omega_{1,3} :+ :\omega_{0,3} \omega_{1,2}:,$$ and the subleading terms can all be expressed as normally ordered monomials in $\omega_{0,n}$ for $n = 1,3,5,7,9$. In fact, the relation can be written as follows.
\begin{equation} \label{weight 12}  \frac{1}{6} \omega_{0,9} = \ 
:\omega_{0,1}\omega_{0,5}:  - 3 :\omega_{0,1} \partial^2 \omega_{0,3}: + 2 : \omega_{0,1} \partial^4 \omega_{0,1}: - 
  : (\partial \omega_{0,1}) \partial^3 \omega_{0,1}: + :(\partial \omega_{0,1}) \partial \omega_{0,3}: \end{equation} $$ + :\omega_{0,3} \partial^2 \omega_{0,1}:  - :\omega_{0,3} \omega_{0,3}:  - \frac{5}{3} \partial^2 \omega_{0,7}  + \frac{19}{2} \partial^4 \omega_{0,5} - \frac{587}{30} \partial^6 \omega_{0,3}   + \frac{119}{10} \partial^8 \omega_{0,1}.$$ This shows that $\omega_{0,9}$ can be expressed as a normally ordered polynomial in $\omega_{0,1}, \dots, \omega_{0,7}$ and their derivatives. Starting with \eqref{weight 12}, and using the fact that $$(\omega_{0,1})_{(1)} (\omega_{0,n}) = -4 (n+5)\omega_{0,n+2} + \nu_n,\qquad n = 1,3,5,\dots,$$ where $\nu_n$ is a linear combination of $\partial^{2} \omega_{0,n}, \partial^4 \omega_{n-2},\dots, \partial^{n+1} \omega_{0,1}$, one can construct decoupling relations $$\omega_{0,n} = Q_n(\omega_{0,1}, \omega_{0,3}, \omega_{0,5}, \omega_{0,7}),\qquad n = 11,13,15,\dots.$$ This shows that $\cG_{\text{ev}}(1)^{Sp(2)}$ is of type $\cW(4,6,8,10)$ with minimal strong generating set $\{\omega_{0,1},\dots, \omega_{0,7}\}$, and is a special case of Theorem \ref{thm:sp-invariant}.

Recall the elements $U_{0,n}$ which satisfy $\lim_{k\ra \infty} \frac{1}{k^2} U_{0,n} = \omega_{0,n}$. The above relation \eqref{weight 12} can be deformed to a relation
$$\lambda(k) U_{0,9} = R_9 (L, U_{0,1}, U_{0,3}, U_{0,5}, U_{0,7}),$$ and it can be verified by a lengthy computer calculation that \begin{equation} \label{lambdakformula} \lambda(k) = \frac{(2+k)(39+14k)(49+18k)}{756(9+2k)}.\end{equation} Note that $$\lim_{k\ra \infty} \frac{1}{k^2} \lambda(k) = \frac{1}{6},$$ as expected by \eqref{weight 12}. It is clear from \eqref{lambdakformula} that the set of values of $k$ where the numerator or denominator of $\lambda(k)$ vanishes lies in $S$. Using the fact that $(U_{0,1})_{(1)} ( U_{0,n}) = f(k,n) U_{0,n+2} + Q_n$, we can construct similar decoupling relations $$U_{0,n} = R_n(L, U_{0,1}, U_{0,3}, U_{0,5}, U_{0,7}),$$ for all $n = 11,13,15,\dots$ and $k \notin S$. This implies 
\begin{thm} For all real numbers $k> -5/2$, $\cC^k(2)$ is of type $\cW(2,4,6,8,10)$ with minimal strong generating set $\{L, U_{0,1},U_{0,3}, U_{0,5}, U_{0,7}\}$.
\end{thm}

\subsection{The case $\gg = \gs\gl_4$} 
The generators of $\cW^k(\gs\gl_4, e_{-\theta})$ are $J, X, Y, H, T, G^{1,\pm}, G^{2,\pm}$, where $T$ is a Virasoro element of central charge $$c = -\frac{3 k (3 + 2 k)}{4 + k},$$ $J,H,X,Y$ are primary of weight one, and $G^{1,\pm}, G^{2,\pm}$ are primary of weight $3/2$. Moreover, $H,X,Y$ generate a copy of $V^{k+1}(\gs\gl_2)$, $J$ commutes with $H,X,Y$ and generates a Heisenberg algebra, and we have the following OPE relations:
$$J(z)J(w) \sim 4(2+k)(z-w)^{-2}, \qquad J(z) X^+(w) \sim 2 X^{+}(w),$$
$$G^{1,-}(z) G^{1,+} (w) \sim -2 (k + 2) X(w)(z-w)^{-2} + \big(:JX:  - (k + 2) \partial X\big)(w)(z-w)^{-1},$$
$$G^{2,-} (z) G^{2,+} (w) \sim -2 (k + 2) Y(w)(z-w)^{-2} + \big( :J Y: - (k + 2) \partial Y\big)(w)(z-w)^{-1},$$ 
$$G^{1,-} (z)G^{2,+} (w) \sim 
-2 (k + 1) (k + 2) (z-w)^{-3} + \big((k + 1) J - (k + 2) H\big)(w)(z-w)^{-2} $$ $$ + \bigg( (4+k) T  - \frac{3}{8} :JJ: + \frac{1}{2} :HJ: - \frac{1}{2} :HH: - 2 :XY: +  \frac{1 + k}{2} \partial J - \frac{k}{2} \partial H \bigg)(w)(z-w)^{-1},$$
$$G^{1,+} (z)G^{2,-}(w) \sim 2 (k + 1) (k + 2) (z-w)^{-3} + \big((k + 1) J+ (k + 2) H\big)(w)(z-w)^{-2} $$
$$ + \bigg( -(4+k) T +\frac{3}{8} :JJ: + \frac{1}{2} :HJ: + \frac{1}{2} :HH: + 2 :XY: +  \frac{1 + k}{2} \partial J  + \frac{k}{2} \partial H \bigg)(w)(z-w)^{-1}.$$
We use the notation $$a^i = \lim_{k \ra \infty} \frac{1}{k} G^{i,+},\qquad b^i = \lim_{k \ra \infty} \frac{1}{k} G^{i,-},\qquad i = 1,2,$$ for the generators of $\cG_{\text{ev}}(2)$, which satisfy
$$a^1(z) b^2(w) \sim 2(z-w)^{-3},\qquad b^1(z) a^2(w) \sim -2(z-w)^{-3}.$$ Note that this normalization is different from the one used earlier. By classical invariant theory, the orbifold $\cG_{\text{ev}}(2)^{GL(2)}$ has strong generators
$$\omega_{i,j} =\ : \partial^i a^1 \partial^j b^2:\  +\ : \partial^j a^2 \partial^i b^1:,\qquad i,j \geq 0,$$ and the smaller set $\{\omega_{0,n} |\ n \geq 0\}$ suffices. Recall that the coset $$\cC^k(4) = \text{Com}(V^{k+1}(\gg\gl_{2}), \cW^k(\gs\gl_4, e_{-\theta}))$$ is generically of type $\cW(2,3,\dots, 14)$ and has Virasoro element 
$$L = T - \frac{1}{8 (2 + k)} :JJ: - \frac{1}{4 (3 + k)} :HH: - \frac{1}{3 + k} :XY: +\frac{1}{2 (3 + k)} \partial H $$ with central charge $$c = -\frac{(1 + k) (3 + 2 k) (8 + 3 k)}{(3 + k) (4 + k)}.$$
Next, consider 
$$u_{0,n} = \ :G^{1,-} \partial^n G^{2,+}  : + :( \partial^n G^{1,+}) G^{2,-}:,\qquad n \geq 0.$$ 
These satisfy $\lim_{k\ra \infty} \frac{1}{k^2} u_{0,n} = \omega_{0,n}$, but do not lie in $\cC^k(4)$ and we would like to find suitable corrections $U_{0,n}$ in $\cC^k(4)$. First, define
$$U_{0,0} = u_{0,0} - \frac{14 + 5 k}{24 (2 + k)^2}  :JJJ: - \frac{1}{2 (2 + k)} :JHH: - \frac{2}{2 + k} :JXY: - \frac{1}{2} :(\partial J)H: $$ $$  -\frac{k}{2 (2 + k)} : J (\partial H): + \frac{4 + k}{2 + k}:TJ: - \frac{16 + 9 k + 2 k^2}{6 (2 + k)} \partial^2 J.$$  It it not difficult to verify that $U_{0,0}$ lies in $\cC^k$ and is primary with respect to $L$.
\begin{lemma} For all $k \geq 0$, we have \begin{equation} (U_{0,0})_{(1)} ( u_{0,n}) = f(n,k) u_{0,n+1} +P_n, \qquad f(n,k) = -\frac{(k+4) (n+4) (5 + k + 3 n + k n)}{(n+1)},\end{equation} where $P_n$ can be expressed as a normally ordered polynomial in $T,J,H,X,Y, u_{0,0},\dots, u_{0,n}$ and their derivatives. \end{lemma}
\begin{proof} This is similar to the proof of Theorem 6.3 of \cite{AL} and is omitted.\end{proof}

Let $$S = \{ -4\} \cup  \{ - \frac{5 + 3 n}{1 + n}|\ n\geq 0\},$$ which is the set of values of $k$ where $f(n,k)$ is either zero or undefined.

We have $(U_{0,0})_{(1)} (U_{0,0}) = f(0,k) u_{0,1} + P_0$. Clearly $\tilde{U}_{0,1} = u_{0,1} + \frac{1}{f(0,k)} P_0$ lies in $\cC^k$ and satisfies $\lim_{k\ra \infty} \frac{1}{k^2} \tilde{U}_{0,1} = \omega_{0,1}$, but $\tilde{U}_{0,1}$ is not the desired correction of $u_{0,1}$ since 
$$(U_{0,0})_{(1)} (\tilde{U}_{0,1}) \neq f(1,k) u_{0,2} + P_1,$$ for any $P_1$ depending only on $J,H,X,Y,T, u_{0,0}$, and their derivatives. The problem is that $P_0$ contains the terms $:T G^{1,-} G^{2,+}  :$ and $:T G^{1,+} G^{2,-}:$ which have the property that $(U_{0,0})_{(1)}  (:T G^{1,-} G^{2,+}  :)$ and $(U_{0,0})_{(1)} (:T G^{1,+} G^{2,-}:)$ have a nonzero coefficient of $u_{0,1}$ when expressed as a normally ordered polynomial in $J,H,X,Y,T, u_{0,0}, u_{0,1}$, and their derivatives. This can be corrected as follows:
$$U_{0,1} = \tilde{U}_{0,1} + :LU_{0,0}:.$$ This has the property that  
$(U_{0,0})_{(1)} (U_{0,1}) = f(1,k) u_{0,2} + P_1$ where $P_1$ depends only on $J,H,X,Y,T, u_{0,0}$, and their derivatives. Moreover, $$(U_{0,0})_{(1)} (U_{0,1}) = U_{0,1} + :L U_{0,0}:.$$

As above, we can construct elements $U_{0,n}$ for all $n\geq 0$ with the property that $$(U_{0,0})_{(1)} (U_{0,n}) = f(k,n) U_{0,n+1} + Q_n,$$ where $Q_n$ is a linear combination of the fields  $$:(\partial^{i_1} L) \cdots (\partial^{i_r} L )(\partial^k U_{0,m}):,\qquad m=0,1,\dots, n,\qquad i_1,\dots, i_r, k\geq 0.$$ Moreover, since $U_{0,n}$ all lie in the algebra generated by $L$ and $U_{0,0}$, all points where the denominator of any term appearing in $U_{0,n}$ vanishes lie in $S$. Also, all structure constants in the OPE of $U_{0,0}(z) U_{0,0}(w)$ have poles lying in $S$. It follows that $\{L, U_{0,n}|\ n \geq 0\}$ closes under OPE for all $k\notin S$, and all poles of the structure constants in the OPEs of these fields lie in $S$. By Remark \ref{mainthm:remark} we obtain 

\begin{thm} For all real numbers $k>-3$, $\cC^k(4)$ is strongly generated by $\{L, U_{0,n}|\ n \geq 0\}$. \end{thm}

Now we consider relations among these generators. Recall that in $\cG_{\text{ev}}(2)^{GL(2)}$, the first normally ordered relation among the generators $\omega_{i,j}$ occurs at weight $15$ and has leading term
$$:\omega_{0,0} \omega_{1,1} \omega_{2,2}: - :\omega_{0,0} \omega_{1,2} \omega_{2,1}: - :\omega_{0,1} \omega_{1,0} \omega_{2,2}: $$ $$ + 
:\omega_{0,1} \omega_{1,2} \omega_{2,0}:  - :\omega_{0,2} \omega_{1,1} \omega_{2,0}: + :\omega_{0,2} \omega_{1,0} \omega_{2,1}:,$$
 which is a normally ordering of a classical determinantal relation in $\text{gr}(\cG_{\text{ev}}(2))^{GL(2)}$. The subleading terms can all be expressed as normally ordered monomials in $\omega_{0,n}$ for $n=0,1,\dots, 12$. In fact, the relation can be rewritten in the form 
 $$\lambda \omega_{0,12} = P(\omega_{0,0}, \omega_{0,1},\dots, \omega_{0,11}),$$ where $P$ is a normally ordered polynomial in $ \omega_{0,0}, \omega_{0,1},\dots, \omega_{0,11}$ and their derivatives, and $\lambda \neq 0$. This relation deforms to a normally ordered relation among the generators $L, U_{0,0}, U_{0,1},\dots, U_{0,12}$ for $\cC^k(4)$, which has the form $$\lambda(k) U_{0,12} = Q(L, U_{0,0}, U_{0,1},\dots, U_{0,11}).$$ Here $\lambda(k)$ is a rational function of $k$ satisfying $\lim_{k\ra \infty} \frac{1}{k^2} \lambda(k) = \lambda$. All coefficients appearing in this relation are rational functions of $k$ whose denominators lie in the set $S$. An immediate consequence is
\begin{thm} For all real numbers $k>-3$ except for the roots of the numerator of $\lambda(k)$, $\cC^k(4)$ is of type $\cW(2,3,\dots, 14)$ and has a minimal strong generating set $\{L, U_{0,n}|\ n= 0,1,\dots, 11\}$.\end{thm}
Unfortunately, it is too difficult to compute the numerator of $\lambda(k)$, although it has finitely many roots. By analogy with the case of $\gs\gp_4$ above, we conjecture that the set of distinct roots of the numerator of $\lambda(k)$ lies in $S$. If so, this implies that $\cC^k(4)$ is of type $\cW(2,3,\dots, 14)$ for all real numbers $k>-3$.

\section{Cosets of $\cW_k(\gg, e_{-\theta})$ at nongeneric levels} \label{sect:simple}

Of much more interest than cosets $\cC^k$ of $\cW^k(\gg, e_{-\theta})$ for generic values of $k$, are cosets $\cC_k$ of the simple quotient $\cW_k(\gg, e_{-\theta})$ for special values of $k$ when $\cW^k(\gg, e_{-\theta})$ is not simple. In \cite{ACL}, one such family was considered, namely, $$\cC_{p/2-3} = \text{Com}(\cH, \cW_{p/2-3}(\gs\gl_3, e_{-\theta})).$$ It was shown in \cite{ArI} that for $p=5,7,9,\dots$, $\cW_{p/2-3}(\gs\gl_3, e_{-\theta})$ is $C_2$-cofinite and rational, and the main result of \cite{ACL} is that $\cC_{p/2-3}$ is isomorphic to the principal, rational $\cW(\gs\gl_{p-3})$-algebra with central charge $c = -\frac{3}{p}(p-4)^2$. Recall that $$\cC^{p/2-3} = \text{Com}(\cH, \cW^{p/2-3}(\gs\gl_3, e_{-\theta}))$$ is of type $\cW(2,3,4,5,6,7)$ for all $p$ as above. Since the natural map $\cC^{p/2-3} \ra \cC_{p/2-3}$ is surjective, this family of $\cW(\gs\gl_{p-3})$-algebras has the following {\it uniform truncation property}; it is of type $\cW(2,3,4,5,6,7)$ for all $p>9$, even though the universal $\cW(\gs\gl_{p-3})$-algebra is of type $\cW(2,3,\dots, p-3)$. In fact, $\cW_{p/2-3}(\gs\gl_3, e_{-\theta})$ is a simple current extension of $V_L \otimes \cW(\gs\gl_{p-3})$, where $V_L$ is the lattice vertex algebra with $L = \sqrt{3p-9}\mathbb{Z}$. This is a surprising coincidence, and one of the goals of this paper is to find other coincidences of this kind.

More generally, suppose that $k$ is a value for which $\cW^k(\gg, e_{-\theta})$ is not simple. Let $\cI_k$ be the maximal proper ideal of $\cW^k(\gg, e_{-\theta})$ graded by conformal weight, so that $$\cW_k(\gg, e_{-\theta}) = \cW^k(\gg, e_{-\theta}) / \cI_k$$ is simple. Let $\gg'\subset \gg^{\natural}_0$ be a simple Lie subalgebra, so that $V(\gg') \cong V^{\ell}(\gg')$ for some $\ell$. Let $\cJ$ denote the kernel of the map $V^{\ell}(\gg') \ra \cW_k(\gg, e_{-\theta})$, and suppose that $\cJ$ is maximal so that $V^{\ell}(\gg') / \cJ\cong L_{\ell}(\gg')$. Finally, let $$\cC^k = \text{Com}(V^{\ell}(\gg'), \cW^k(\gg, e_{-\theta})),\qquad \cC_k = \text{Com}(L_{\ell}(\gg'), \cW_k(\gg, e_{-\theta})).$$ By Theorem 8.1 of \cite{CLII}, if $\ell+h^\vee$ is a positive real number, where $h^{\vee}$ is the dual Coxeter number of $\gg'$, then $\pi_k: \cC^k \ra \cC_k$ is surjective. A similar result holds if $\gg'$ is any reductive Lie subalgebra such that this condition holds for each simple ideal. If $\pi_k$ is surjective and we know strong generators for $\cC^k$, they will descend to strong generators for $\cC_k$. This shows that determining both the generic behavior of $\cC^k$ and the structure of the nongeneric set provides a powerful approach to studying $\cC_k$. Using this method, we describe $\cC_k$ in a few interesting examples, and we conjecture some new coincidences of the above kind.

\subsection{Type $C$ series at positive half-integer levels}
Let $\gg = \gs\gp_{2n}$ for $n\geq 2$, and let $k$ be a half-integer such that $k+1/2$ is a positive integer. 
It was shown in \cite{ArIII} that the associated variety of $L_k(\gg)$ is the minimal nilpotent orbit closure of $\fing$, and hence \cite{ArIII}, $\cW^k(\gs\gp_{2n}, e_{-\theta})$ is $C_2$-cofinite.
By \cite{DM2} this implies that the maximal proper graded ideal $\cI_k \subset \cW^k(\gs\gp_{2n}, e_{-\theta})$ contains $:(J^e)^{k+3/2}:$, where $e \in\gs\gp_{2n-2}$ denotes the highest root vector. Therefore we have an embedding $$L_{k+1/2}(\gs\gp_{2n-2}) \hookrightarrow \cW_k(\gs\gp_{2n}, e_{-\theta}).$$ By Corollary \ref{simplicity}, 
$$\cC_k(n) = \text{Com}(L_{k+1/2}(\gs\gp_{2n-2}), \cW_k(\gs\gp_{2n},e_{-\theta}))$$ is simple. 
The vertex algebra $\cW_k(\gs\gp_{2n}, e_{-\theta})$ is conjecturally rational \cite{ArIII}, so one expects that $\cC_k(n)$ is $C_2$-cofinite and rational as well. However, since $\cW_k(\gs\gp_{2n}, e_{-\theta})$ is not generally a simple current extension of $L_{k+1/2}(\gs\gp_{2n-2}) \otimes \cC_k(n)$, this is out of reach at the moment.

We now specialize to the case $n=2$. Recall that for all real numbers $k> -5/2$, $\cC^k(2)$ is of type $\cW(2,4,6,8,10)$ and the map $\pi_k: \cC^k(2) \ra \cC_k(2)$ is surjective. We use the same notation $\{L, U_{0,1}, U_{0,3},U_{0,5},U_{0,7}\}$ for the images of the generators in $\cC_k(2)$. Since $\pi_k$ is surjective, this set strongly generates $\cC_k(2)$, but it need not be minimal.

It was shown by Kawasetsu \cite{Ka} that $\cC_{1/2}(2)$ is isomorphic to the rational Virasoro algebra $L(-25/7,0)$. Here we give an alternative proof of this result. It is easy to verify that 
$U_{0,1} + \frac{132}{5}:LL:  - \frac{21}{2} \partial^2 L$ lies in the ideal $\cI_{1/2} \subset \cC^{1/2}(2)$, so we have the relation $$U_{0,1} = - \frac{132}{5}:LL:  + \frac{21}{2} \partial^2 L$$ in $\cC_{1/2}(2)$. By applying $(U_{0,1})_{(1)} $ successively to this relation, and using the fact that 
$$(U_{0,1})_{(1)} (U_{0,n}) = \frac{14 (5 + n) (4 + 3 n) (8 + 3 n)}{5 (1 + n) (2 + n)} U_{0,n+2} + Q_n,\qquad n = 1,3,5,\dots,$$
we obtain relations $$U_{0,n} = P_n(L),\qquad n = 3,5,7$$ in $\cC_{1/2}(2)$, where  $P_n(L)$ is a normally ordered polynomial in $L$ and its derivatives. Therefore $\cC_{1/2}(2)$ is strongly generated by $L$, and since $\cC_{1/2}(2)$ is simple we recover Kawasetsu's result.

\begin{thm} $\cC_{3/2}(2)$ is isomorphic to the rational $\cW(\gs\gp_4,f_{\text{prin}})$-algebra with central charge $c = -49/6$.
\end{thm} 

\begin{proof} In weight $6$, it can be checked by computer that the element  
$$ U_{0,3} - \frac{39664}{1701}  :LLL: + \frac{117533}{5103} :(\partial^2 L)L: +\frac{1679}{972} :(\partial L)(\partial L): $$ $$- \frac{3}{2} :L U_{0,1}:  + \frac{34801}{20412} \partial^4  L    +  \frac{37}{4536} \partial^2 U_{0,1}$$ in $\cC^{3/2}(2)$ lies in the ideal $\cI_{3/2}$. Therefore the corresponding element in $\cC_{3/2}(2)$ is a relation expressing $U_{0,3}$ as a normally ordered polynomial $P_3(L, U_{0,1})$  in $L,U_{0,1}$ and their derivatives. 
By applying $(U_{0,1})_{(1)}$ successively to this relation, and using the fact that $$(U_{0,1})_{(1)} (U_{0,n}) = \frac{6 (5 + n) (5 + 2 n) (5 + 4 n)}{(1 + n) (2 + n)} U_{0,n+2} + Q_n,$$ we obtain relations $U_{0,5} = R_5(L,U_{0,1})$ and $U_{0,7} = R_7(L, U_{0,1})$, where $R_5$ and $R_7$ are normally ordered polynomials in $L, U_{0,1}$ and their derivatives. Therefore $\cC_{3/2}(2)$ is of type $\cW(2,4)$ with strong generators $\{L, U_{0,1}\}$. Checking that it is actually a $\cW(\gs\gp_4,f_{\text{prin}})$-algebra is straightforward by computer using the explicit formulas for $\cW(\gs\gp_4,f_{\text{prin}})$ appearing in \cite{ZhII}. Since $\cC_{3/2}(2)$ is simple and the simple quotient of $\cW(\gs\gp_4,f_{\text{prin}})$ at this central charge is rational (\cite{ArIII}, see below), the claim follows. \end{proof}

We thus have that $\cW_{3/2}(\gs\gp_4, e_{-\theta})$ is an extension of a rational VOA. The VOA extension problem has a precise categorical formulation and especially if the representation category of a VOA is semi-simple, then this is also true for any extension provided the extended VOA is a simple VOA \cite{HKL,CKM}. In other words:
\begin{cor} $\cW_{3/2}(\gs\gp_4, e_{-\theta})$ is rational.
\end{cor}

The central charge of $\cC_k(n)$ is $$-\frac{(1 + 2 k) (2 + 2 k + n) (2 + 3 k + 2 n)}{(1 + k + n) (1 + 2 k + 2 n)}.$$ This is the same as the central charge of the principal $\cW$-algebra $\cW_s(\gs\gp_{2m},f_{\text{prin}})$, where $$m=k+1/2,\qquad s+(k+3/2)=p/q,\qquad (p,q)=(n+k+1/2,2n+2k+2).$$
Since $q=2(n+k+1/2)+1$ is odd,
$p\geq m+1$ and $q\geq 2m$,
the level $s$ is a nondegenerate admissible number, so that
$\cW_s(\gs\gp_{2m},f_{\text{prin}})$ is rational and $C_2$-cofinite \cite{ArIII,ArII}.

Based on this observation, we make the following conjecture.
\begin{conj} For all $n\geq 2$ and $k$ such that $k+1/2$ is a positive integer, $\cC_k(n)$ is isomorphic to the $C_2$-cofinite, rational principal $\cW$-algebra
$\cW_s(\gs\gp_{2m},f_{\text{prin}})$ as above.
\end{conj}

Note that this conjecture would imply the rationality of $\cW_{k+1/2}(\gs\gp_{2n}, e_{-\theta})$ for all $n\geq 2$ and all $k$ such that $k+1/2$ is a positive integer.

\subsection{Type $A$ series at positive integer levels}
Let $\gg = \gs\gl_n$ for $n\geq 4$, and let $k$ be a non-negative integer. 
By Lemma 10.3 of \cite{AMsheet}
the maximal proper graded ideal $\cI_k\subset \cW^k(\gs\gl_n, e_{-\theta})$ contains $:(J^e)^{k+2}:$, where $e \in\gs\gl_{n-2}$ denotes the highest root vector. Therefore we have an embedding $$ L_{k+1}(\gg\gl_{n-2}) = \cH \otimes L_{k+1}(\gs\gl_{n-2}) \hookrightarrow \cW_k(\gs\gl_n, e_{-\theta}),$$ and by Corollary \ref{simplicity},
$$\cC_k(n) = \text{Com}(L_{k+1}(\gg\gl_{n-2}), \cW_k(\gs\gl_n, e_{-\theta}))$$ is simple. Since the projection $\pi_k: \cC^k(n) \ra \cC_k(n)$ is surjective, $\cC_k(n)$ is the simple quotient of $\cC^k(n)$.

The $\cW$-algebra $\cW_k(\gs\gl_n, e_{-\theta})$ should be \lq\lq small" in the following sense.
\begin{conj} \label{conj.dim1}
For all integers $n\geq 4$ and $k\geq -1$, 
 the associated variety of
$\cW_k(\gs\gl_n, e_{-\theta})$ is isomorphic to $\mathbb{A}^1$.
In particular, it is one-dimensional. \end{conj}
Here the associated variety $X_V$ of a vertex algebra $V$ is defined as
$X_V=\on{Specm}(R_V)$,
where $R_V$ is the Zhu's $C_2$-algebra of $V$ \cite{Ara12}.
Conjecture \ref{conj.dim1}
holds for $k=-1$ by Theorem 7.4 of \cite{AMsheet}.

We specialize to the case $n=4$. Recall that for all real numbers $k>-3$, $\cC_k(4)$ is strongly generated by the fields $\{L, U_{0,n}|\ n\geq 0\}$ and $\pi_k: \cC^k(4) \ra \cC_k(4)$ is surjective. 

\begin{thm}
\label{thm:c04}
 $\cC_0(4)$ is isomorphic to the simple Zamolodchikov $\cW_3$-algebra with $c=-2$ \cite{Zam}, which we denote by $\cW_{3,-2}$. \end{thm}
\begin{proof} It is not difficult to verify by computer that the element 
$$U_{0,1} + \frac{16}{5}:LL:  - \frac{3}{5} \partial^2 L$$ in $\cC^0(4)$ lies in the ideal $\cI_0$, and hence gives rise to the relation $U_{0,1} = - \frac{16}{5}:LL:  + \frac{3}{5} \partial^2 L$ in $\cC_0(4)$. By applying $(U_{0,0})_{(1)}$ successively to this relation, and using the fact that $$(U_{0,0})_{(1)} (U_{0,n}) = -\frac{4 (4 + n) (5 + 3 n)}{1 + n} U_{0,n+1} + Q_n,$$ we obtain relations $$U_{0,n} = P_n(L,U_{0,0}),\qquad n\geq 0.$$ It follows that $\cC_0(4)$ is strongly generated by $\{L, U_{0,0}\}$. Finally, it is not difficult to check by computer that they satisfy the OPE relations of $\cW_{3,-2}$, modulo the ideal $\cI_0$.
\end{proof}

\begin{remark} 
\label{rem:c04}
$\cC_0(4)$ has a one-dimensional associated variety, and admits a well-known extension called a triplet algebra which is $C_2$-cofinite but non-rational. \end{remark}

\begin{thm} \label{thm:c14} $\cC_1(4)$ is isomorphic to the simple parafermion algebra $$K_{-6/5}(\gs\gl_2) = \text{Com}(\cH, L_{-6/5}(\gs\gl_2)),$$ which has central charge $c=-11/2$.
\end{thm}
\begin{proof} Note first that if $\cV$ is a vertex algebra with strong generating set $\{\omega_1,\dots,\omega_r\}$, the OPE algebra among the generators does not determine $\cV$ uniquely. Instead, this OPE algebra determines a {\it category} of vertex algebras. If $\cV$ is $\mathbb{Z}_{\geq 0}$-graded by weight and has finite-dimensional weight spaces, there is a unique {\it simple} vertex algebra in this category. It is characterized by the property that it is a homomorphic image of any other vertex algebra in the category. In particular, if two simple $\mathbb{Z}_{\geq 0}$-graded vertex algebras with finite-dimensional weight spaces have the same strong generating set and OPE algebra, they must be isomorphic.

It is not difficult to find the generators of weight $2,3,4,5$ of the universal parafermion algebra $K^{-6/5}(\gs\gl_2)$ and compute all OPEs among these generators by computer. Clearly the simple parafermion algebra $K_{-6/5}(\gs\gl_2)$ has the same strong generators and OPEs, although there are now additional normally ordered relations among the generators.

It can be checked by computer that in the ideal $\cI_1 \subset \cC^1(4)$, there is an element of weight $6$ of the form
$$U_{0,3} - P_3(L, U_{0,0}, U_{0,1}, U_{0,2}),$$ where $P_3$ is a normally ordered polynomial in $L, U_{0,0}, U_{0,1}, U_{0,2}$, and their derivatives. Therefore in $\cC_1(4)$, there is a relation of weight $6$ of the form 
$$U_{0,3} = P_3(L, U_{0,0}, U_{0,1}, U_{0,2}).$$  Applying $(U_{0,0})_{(1)}$ successively to this relation, we obtain relations $$U_{0,n} = P_n(L, U_{0,0}, U_{0,1},U_{0,2}),\qquad n\geq 0,$$ so $\cC_1(4)$ is of type $\cW(2,3,4,5)$ with strong generators $\{L, U_{0,0}, U_{0,1},U_{0,2}\}$. It is then a straightforward but lengthy computer calculation to verify that these generators satisfy the OPE relations of $K^{-6/5}(\gs\gl_2)$, modulo the ideal $\cI_1$. Finally, since there is a unique simple vertex algebra of type $\cW(2,3,4,5)$ with this OPE algebra, the claim follows.
\end{proof}

\begin{remark}
One might guess that $K_{-6/5}(\gs\gl_2)$ is isomorphic to the algebra $\cW(\gs\gl_5, f_{\text{prin}})$ with $c = -11/2$. However, this turns out to be false, which can be shown using the explicit OPE relations for $\cW(\gs\gl_5, f_{\text{prin}})$ given in \cite{ZhI}. 

Aspects of the representation theory of parafermion algebras of $\gs\gl_2$ at admissible levels, including $K_{-6/5}(\gs\gl_2)$, are studied in \cite{ACR}. They support Conjecture \ref{conj.dim1}, and modularity results also indicate that they allow for $C_2$-cofinite VOA-extensions. 
\end{remark}

For $n\geq 3$, let $\cW^{\ell}(\gs\gl_n,f_{\text{subreg}})$ be the $\cW$-algebra associated with a subregular nilpotent element $f_{\text{subreg}}$
at level $\ell$ \cite{KRW}. It is freely generated of type $\cW(1,2,\dots, n-1, n/2, n/2)$ and the weight one field generates a Heisenberg algebra $\cH$. The central charge is $$c = -\frac{((\ell +n)(n-1)-n)((\ell +n)(n-2)n - n^2+1)}{\ell+n}.$$ It was recently shown by Genra \cite{G} that 
$\cW^{\ell}(\gs\gl_n,f_{\text{subreg}})$ is isomorphic to 
the vertex algebra
$\cW^{(2)}_{n,\ell}$ introduced by 
Feigin and Semikhatov \cite{FS}. Note that for $n=1$ and $n=2$, $\cW^{(2)}_{n, \ell}$ is well-defined and coincides with the rank one $\beta\gamma$-system and the affine vertex algebra $V^{\ell}(\gs\gl_2)$, respectively.
We denote by $\cW_{\ell}(\gs\gl_n,f_{\text{subreg}})$ the unique simple quotient of $\cW^{\ell}(\gs\gl_n,f_{\text{subreg}})$.

Observe that $\cC_k(n)$ has central charge $$c = -\frac{(1 + k) (-1 + 2 k + n) (3 k + 2 n)}{(-1 + k + n) (k + n)},$$ which coincides with the central charge of $\text{Com}(\cH, \cW_{\ell}(\gs\gl_{k+1},f_{\text{subreg}}))$ for $\ell = - \frac{1 + k^2 + k n}{k + n}$. In the cases $k=0$ and $k=1$, this coincidence still holds where $\cW_{\ell}(\gs\gl_{k+1},f_{\text{subreg}})$ is now replaced by the $\beta\gamma$-system and $L_{\ell}(\gs\gl_2)$, respectively. Based on this observation, and the belief that vertex algebras with this central charge and one-dimensional associated variety are rare, we make the following conjecture.

\begin{conj} \label{typeAconj} For all integers $n\geq 4$ and $k\geq 0$, 
such that $n^2-k^2-1\geq n$,
$$\cC_k(n) \cong \text{Com}(\cH, \cW_{\ell}(\gs\gl_{k+1},f_{\text{subreg}})),\qquad \ell = - \frac{1 + k^2 + k n}{k + n}.$$ For $k=0$ and $k=1$, we replace $\cW_{\ell}(\gs\gl_{k+1},f_{\text{subreg}})$ above with the $\beta\gamma$-system and the simple affine vertex algebra $L_{\ell}(\gs\gl_2)$, respectively. 
\end{conj}
 
\begin{remark}
Conjecture \ref{typeAconj} fails if 
$k$ is large relative to $n$. Suppose that $n^2-k^2-1\geq n$.
As $n^2-k^2-1$ and $k+n$ are mutually prime,
$\ell$ is a nondegenerate  admissible number,
that is, the associated variety of $L_{\ell}(\mathfrak{sl}_n)$ is the nilpotent cone of $\mathfrak{sl}_n$ \cite{ArIII}.
Therefore, the associated variety of the simple Feigin-Semikhatov algebra at level $\ell$, or
the subregular $\cW$-algebra of $\gs\gl_n$ at level $\ell$, is the intersection of the nilpotent cone with the Slodowy slice at a subregular nilpotent element, which is isomorphic to the $A_n$-singularity \cite{ArIII}. In particular, it has the symplectic singularity and has dimension $2$.
\end{remark}

Since $\cW^{(2)}_{1,\ell}$ is isomorphic to the rank one $\beta\gamma$-system $\cS$ for all $\ell$, and $\text{Com}(\cH, \cS)\cong \cW_{3,-2}$ \cite{Wa}, this conjecture holds for $n=4$ and $k=0$ by Theorem \ref{thm:c04}. Similarly, since the simple Feigin-Semikhatov algebra $\cW^{(2)}_{2,\ell} \cong L_{\ell}(\gs\gl_2)$, our conjecture holds for $n=4$ and $k=1$ by Theorem \ref{thm:c14}. In the case, $n=4$ and $k=2$, $\text{Com}(\cH, \cW_{\ell}(\gs\gl_3, f_{\text{subreg}}))$ is of type $\cW(2,3,4,5,6,7)$ \cite{ACL}, and many (but not all) OPE relations have been checked by computer to coincide with the OPE relations in $\cC_2(4)$.

We now prove a more general statement using Theorem \ref{uniqueness} that implies our conjecture for $k=0$ and all $n\geq 4$. Let $\cE(n)$ be the rank $n$ $bc$-system with generators $b^i, c^i$, and let $\cS(m)$ be the rank $m$ $\beta\gamma$-system with generators $\beta^i, \gamma^i$. Let $\cA(1)$  be the rank one symplectic fermion algebra with generators $X^{\pm}$. Define a $U(1)$ action on $\cE(n) \otimes \cS(m) \otimes \cA(1)$ by 
\[
X^\pm \mapsto \lambda^{\pm 1} X^\pm, \qquad b^i \mapsto \lambda b^i, \qquad c^i \mapsto \lambda^{-1} c^i , \qquad \beta^j \mapsto \lambda \beta^j, \qquad \gamma^j \mapsto \lambda^{-1}\gamma^j
\]
Here the $1+n+m$ vectors $\{X^+,b^i, \beta^j\}$ and $\{X^-, c^i, \gamma^j \}$ should be viewed as carrying the standard and mutually conjugate representations of $\gg\gl(n+1|m)$, respectively. With this notation, we have the following result.
\begin{thm}
For all $(n, m)\neq (0,2)$,
$\cW_0\left(\gs\gl(n+2|m),e_{-\theta} \right) \cong \left(\cE(n) \otimes \cS(m) \otimes \cA(1)\right)^{U(1)}$. 
\end{thm}

\begin{proof}

Let $(n, m)\neq (0,2)$. It is well-known that $\left(\cE(n) \otimes \cS(m)\right)^{U(1)}\cong L_1\left(\gg\gl(m|n)\right)$. This was first shown in \cite{KWIII} for $m>0$, and then proved in the cases $n>2$ and $m=0$ \cite{AP}. An alternative proof can be found in the recent paper \cite{CKLR}.

Furthermore, $\cA(1)^{U(1)}$ is well known to be Zamolodchikov's $\cW_3$ algebra $\cW_{3,-2}$. It is clear that the generators of $L_1(\gg\gl(m|n))$ and $\cW_{3,-2}$, together with the fields $:c^iX^+:$ and $:b^i X^-:$, form a strong generating set for $\left(\cE(n) \otimes \cS(m)\right)^{U(1)}$. Also, it is easy to see that the dimension $3$ field of $\cW_{3,-2}$ (one can take $:X^+\partial X^-:$ for it) is a normally ordered polynomial in the other strong generators and their derivatives. 
So all conditions for Theorem \ref{uniqueness} are satisfied.
Finally, by \cite{DLM} the orbifold is simple.
\end{proof}
\begin{remark} If $(n,m) =(0,2)$ then the orbifold $\cS(2)^{U(1)}$ is larger than just $L_{-1}(\gg\gl_2)$. Since $\gg\gl(2|2)$ is not reductive, it is better to consider the simple quotient $\gp\gs\gl(2|2)$ of $\gg\gl(2|2)$. In that case, the minimal $\cW$-algebra is the small $N=4$ super Virasoro algebra and constructions of it are given in \cite{CKLR, Ad}.\end{remark}

\begin{cor} Let $$\cD_0(n,m) = \text{Com}(L_1(\gg\gl(n|m)), \cW_0\left(\gs\gl(n+2|m),e_{-\theta} \right)).$$ Then for all $(n,m) \neq (0,2)$, $$\cD_0(n,m) \cong \cA(1)^{U(1)} \cong \cW_{3,-2}.$$ In particular, $\cC_0(n) = \cD_0(n-2,0) \cong \cW_{3,-2}$ for all $n\geq 4$, so Conjecture \ref{typeAconj} holds for all $n\geq 4$ and $k=0$.
\end{cor}

\end{document}